\documentclass[12pt]{amsart}

\usepackage[english]{babel}
\usepackage[utf8x]{inputenc}
\usepackage{geometry,amsfonts,amssymb,amsthm,txfonts,pxfonts,amscd} 
\usepackage{amsmath,amscd}
\usepackage{graphicx,caption}
\usepackage[colorinlistoftodos]{todonotes}
\usepackage{algorithmicx,algpseudocode}
\usepackage{algorithm}
\usepackage{mathabx}
\usepackage{refcheck,graphicx,url}
\usepackage{listings}
\norefnames
\nocitenames
\def\struckint{\mathop{%
\def\mathpalette##1##2{\mathchoice{##1\displaystyle##2}%
  {##1\textstyle##2}{##1\scriptstyle##2}{##1\scriptscriptstyle##2}}%
\mathpalette
{\vbox\bgroup\baselineskip0pt\lineskiplimit-1000pt\lineskip-1000pt
\halign\bgroup\hfill$}
{##$\hfill\cr{\intop}\cr\diagup\cr\egroup\egroup}%
}\limits}
\newtheorem{theorem}{Theorem}[section]

\newtheorem{corollary}[theorem]{Corollary}
\newtheorem{conjecture}[theorem]{Conjecture}

\newtheorem{fact}[theorem]{Fact}

\theoremstyle{remark}

\newtheorem{remark}[theorem]{Remark}

\newtheorem{question}[theorem]{Question}

\newcommand{\integers}{\mathbb{Z}}

\newcommand{\reals}{\mathbb{R}}

\DeclareMathOperator{\mcg}{Mod}

\DeclareMathOperator{\Sp}{Sp}
\DeclareMathOperator{\SL}{SL}

\DeclareMathOperator{\coker}{coker}

\DeclareMathOperator{\Volume}{Volume}

\DeclareMathOperator{\area}{area}
\DeclareMathOperator{\height}{height}

\title{Statistics of 3-manifolds occasionally fibering over the circle}
\author{Igor Rivin}
\address{Mathematics Department, Temple University, Philadelphia, PA}
\curraddr{Mathematics Department, Brown University and ICERM, Providence, RI}
\email{rivin@temple.edu}
\begin{document}

\begin{abstract}
We study random elements of subgroups (and cosets) of the mapping class group of a closed hyperbolic surface, in part through the properties of their mapping tori. In particular, we study the distribution of the homology of the mapping torus (with rational, integer, and finite field coefficients, the hyperbolic volume (whenever the manifold is hyperbolic), the dilatation of the monodromy, the injectivity radius, and the bottom eigenvalue of the Laplacian on these mapping tori. We also study mapping tori of punctured surface bundles, and various invariants of their Dehn fillings. We also study corresponding questions  in the Dunfield-Thurston (\cite{dunfield2006finite}) model of random Heegard splittings of fixed genus, and give a number of new and improved results in that setting.
\end{abstract}
\thanks{The author would like to thank Jeff Brock, Nathan Dunfield, Mark Bell, Saul Schleimer, Anders Karlsson, Alex Eskin, Juan Souto, Ian Biringer, Kasra Rafi, Stefan Friedl, and David Futer for useful conversations, Richard Stanley for the nice proof of Theorem \ref{snfthm} and ICERM and Brown University for their generous support and hospitality during the writing of this paper.The very extensive computational experiments would not have been possible without generous support. I also would like to thank Emmanuel Kowalski for helpful comments on the first version of this preprint. The computations discussed were performed using Mathematica, SnapPy, and Twister \cite{Twister}. A previous version of this paper (containing most of the results for fibered manifolds) was circulated as the arXiv.org preprint \cite{rivin2014statistics}}
\maketitle
\tableofcontents
\section{Introduction}

The modern study of random $3$-manifolds, and their relationship to random elements of the mapping class group was initiated by N.~Dunfield and W.~Thurston in their foundational paper \cite{dunfield2006finite} (they studied ``random Heegaard splittings''). Their work spurred on considerable progress in the field, which lies at the crossroads of three-manifold topology, surface topology, dynamics, ergodic theory, number theory, algebraic groups, and the study of Kleinian groups. In addition, while Dunfield and Thurston were motivated by the Virtual Haken Conjecture, in the interim Agol, Groves and Manning \cite{agol2013virtual} extended the work of Dani Wise and collaborators (see, e.g., \cite{hagwise}) to resolve the (much stronger) Virtual Fibering Conjecture, which makes $3$-manifolds fibering over $S^1$ even more interesting than in the past.  With this in mind, in this paper we undertake the study of random manifolds fibering over $S^1,$ and discover that much can be said, and much remains mysterious.

Let $S$ be a surface of genus $g,$ and let $\phi:S\righttoleftarrow$ be an automorphism of $S.$ For such a map $\phi,$ we can define the mapping torus $M_\phi,$ which will be a three-dimensional manifold fibering over the circle $S^1.$

It is clear that the properties of $\phi$ are reflected in the topology of $M_\phi,$ and, consequently, an understanding of the properties of ``random'' surface automorphisms leads to an understanding of ``random'' manifolds fibering over the circle. The word ``random'' is in quotes because there are a number of possible ways we can define what it means -- see the author's survey \cite{rivinMSRI}. In this note we will define a random element of the mapping class group (or a subgroup thereof) as a "long" random product of generators (the generating set is assumed to be symmetric), though all of our results hold with a somewhat more general definition using symmetric recognizing automata for permitted words.

Here is a list of some of our observations, we will denote the set of random products of length $n$ by $W_n.$ In fact, it is not necessary we take random products of generators of the whole mapping class group, so the list below is only to give the flavor of the results we find. In all statements, $\phi$ is a uniformly chosen random element of $W_n.$
\begin{itemize}
\item The map $\phi$ is pseudo-Anosov, with probability approaching $1$ exponentially fast in $n.$
\item The mapping torus $M_\phi$ is a hyperbolic manifold with probability approaching $1$ exponentially fast in $n.$
\item The hyperbolic volume of the mapping torus $M_\phi$ grows linearly in $n.$
\item The first Betti number of the mapping torus $M_\phi$ of $\phi$ is $1,$ with probability approaching $1$ exponentially fast in $n.$
\item The expected value of the logarithm of the order of the torsion subgroup of $H_1(M_\phi, \integers)$ grows linearly in $n.$ Furthermore, this logarithm is asymptotically normally distributed.
\item The probability that the rank of $H_1(M_\phi, \integers/p \integers)$ for a prime $p$ is greater than $k$ asymptotically approaches $c_k/p^k,$ where $c_k$ are universal constants. 
\item The rank of the fundamental group of $M_\phi$ is generically $2g+1.$
\item The injectivity radius of $M_\phi$ goes to zero with $n.$ 
\item The smallest positive eigenvalue $\lambda_1(M_\phi)$ goes to zero between linearly and quadratically as a function of $n.$
\end{itemize}

We also study manifolds which fiber over the circle with fiber a cusped surface. We study the generic behavior of the cusp of the resulting (generically hyperbolic) manifold, and use it to show that for a generic such manifold \emph{all} the Dehn fillings are hyperbolic. In addition, by varying the model somewhat, we get large families of \emph{integer homology spheres} with control on the Casson invariants.
We also make the following conjectures (coming from computational experiments):
\begin{itemize}
\item The expected hyperbolic volume of a surface bundle grows linearly in the length of the monodromy as a word in the mapping class group.
\item The expected growth rates of both volume and logarithm of the torsion subgroup of homology of a mapping torus approach \emph{finite} limits as the genus of the fiber becomes large.
\end{itemize}

We also give a number of results on the Dunfield-Thurston model of random Heegaard splittings.
\begin{itemize}
\item The cardinality of the first homology group of a random Heegard splitting grows exponentially with the length of the gluing map, and there are precise limit theorems for its size (this is true if the gluing map is chosen from a nonelementary subgroup of the mapping class group, whose Torelli image is Zariski dense in $\Sp(2g, \integers).$
\item Under the same assumptions as above, the \emph{Kneser-Matve'ev complexity} of the manifold grows linearly.
\item Under the same assumptions, the probability that the manifold has a Galois cover with Galois group $Q$ (which is assumed to be solvable) and positive first Betti number goes to zero with the length of the gluing map.
\item The volume of a random Heegard splitting goes to infinity linearly  with the length of the defining word (resolving a conjecture of Dunfield-Thurston), while
\item The Cheeger constant, the injectivity radius, and the bottom eigenvalue of the Laplacian go to zero with the length of the word (this is true for any nonelementary subgroup).
\end{itemize}
We also conjecture that the probability that the random manifold is Haken goes to zero with the length of the gluing map.

The plan of the rest of the paper is as follows:

In Section \ref{background} we will recall some background.

In Section \ref{matprods} we study the homology of random mapping tori.

In Section \ref{volsec} we study the hyperbolic volumes and other geometric invariants of random mapping tori.

In Section \ref{prehom} we find many hyperbolic manifolds fibering over the circle with bounded volume and prescribed structure of the first integral homology group.

In Section \ref{randsubgp} we discuss the structure of a random subgroup of the mapping class group.

In Section \ref{cusped} we study fibered manifolds with punctured fiber, and their Dehn fillings.

In Section \ref{experiments} we will describe some experimental results and conjectures.

In Section \ref{randomh} we address random Heegaard splittings \textit{a la} Dunfield-Thurston.

\section{Some background}
\label{background}

In this section we will recall some (fairly diverse) background facts.
\subsection{Homology}
\label{homosec}
The following result is shown (as Example 2.48) in A. Hatcher's Algebraic Topology \cite{hatcher2002algebraic}:
\begin{fact}[\cite{hatcher2002algebraic}[Example 2.48]
Let $X$ be a space, $\phi: M \righttoleftarrow$ be a self-map of $X,$ and let $M_\phi$ be the mapping torus of $\phi.$ Then, we have the following exact sequence:
\[
\begin{CD}
\dots \rightarrow H_n(X) @>\phi_* - \iota_*>> H_n(X) @>\iota_*>> H_n(M_\phi) @>>> H_{n-1}(X) @>\phi_* - \iota_*>> H_{n-1}(X) \rightarrow \dots,
\end{CD}
\]
where $\iota$ is the identity map, and the lower $*$ indicates the induced map on homology.
\end{fact}
In the case of interest to us, $X$ is closed connected surface, so letting $n=1,$ we have $H_0(X)\simeq \integers,$ and the map $\phi_*$ is identity on zeroth homology, so we have:
\[
\begin{CD}
\dots \rightarrow H_1(X,\integers) @>\phi_* - \iota_*>> H_1(X,\integers) @>\iota_*>> H_1(M_\phi,\integers) @>>> \integers @>>> 0
\end{CD}
\]
Since $\integers$ is a free $\integers$ module, we have the following corollary:
\begin{corollary}
\label{mappinghomo}
If $X$ is a connected topological space, $\phi: X\righttoleftarrow$ a self-map of $X,$ and $M_\phi$ is the mapping torus $M_\phi,$ then
\[
H_1(M_\phi, \integers) \simeq \coker (\phi_* - \iota_* )\bigoplus \integers.
\]
\end{corollary}
\begin{corollary}
\label{highrank}
The rank of homology $H_1(M_\phi, F)$ with coefficients in a field $F$ equals $1$ plus the dimension of the eigenspace with eigenvalue $1$ of $\phi_*.$
\end{corollary}
\begin{corollary}
If the map $\phi_* - \iota_*$ is nonsingular, then the order of the torsion subgroup of $H_1(M_\phi, \integers)$ equals $|\det (\phi_* - \iota_*)|.$
\end{corollary}
In the sequel we will need the following result, which can be proved in a number of ways.
\begin{theorem} 
\label{snfthm}
Let $\alpha_1 | \alpha_2 | \dots |\alpha_{2n}$ be integers. Then, there exists a symplectic matrix $M,$ such that the elementary divisors of $M-I$  are $\alpha_1, \dotsc, \alpha_{2n}.$
\end{theorem}
\begin{proof} 
For $n=1,$ $\Sp(2, \integers) \simeq \SL(2, \integers).$ Suppose the elementary divisors are $\alpha_1 = r, \alpha_2 = rs.$ Then the matrix
\[
\begin{pmatrix}
1-r(1+rs) & r\\
-r(1+s+rs) & 1+r
\end{pmatrix}
\]
is the required matrix.\footnote{The author is indebted to R.P. Stanley for this construction} Now, for $n>1$ we construct a block-diagonal matrix with required elementary divisors.
\end{proof}
\subsection{Geometric structures}
\label{gstructs}
The foundational result for mapping tori is Thurston's theorem:
\begin{theorem}[Thurston's geometrization theorem for 3-manifolds fibering over the circle]
\label{thurstonfiber}
Let $\phi$ be a pseudo-anosov automorphism of a finite area hyperbolic surface $S.$ Then, $M_\phi$ admits a unique (by the Mostow-Prasad rigidity theorem) complete hyperbolic structure of finite volume.
\end{theorem}

The obvious question raised by (already highly non-trivial) Theorem \ref{thurstonfiber} is: what can we say about the hyperbolic geometry of $M_\phi$ if we have some information about $\phi.$ The simplest (in some ways) invariant of a hyperbolic manifold is its volume, and the basic result is Jeff Brock's theorem:
\begin{theorem}[Jeff Brock \cite{brocktori}]
For a pseudo-anosov Map $\phi,$ let $\|\phi\|$ be the translation distance of $\phi$ on Teichm\"uller space equipped with the Weil-Peterson metric. Then, for some universal constant $K > 0,$ we have:
\[
\frac{1}{K} \|\phi\| \leq \Volume(M_\phi) \leq K \|\phi\|.
\]
\end{theorem}
\begin{remark}
\label{curvecomplexes}
As Brock had also shown (see \cite{brockjams}), the Weil-Peterson metric is coarsely isometric to the pants complex distance, which, in turn, majorizes the curve complex distance.
\end{remark}

\section{Random matrix products and homology}
\label{matprods}
Consider a real semisimple algebraic group $G,$ and consider products of elements of the set $S=\{M_1, \dotsc, M_k\}.$ Consider a random such product $P$ of length $n.$ Using the $KAK$ decomposition (known in the numerical linear algebra circles as the Singular Value Decomposition), we can ask about the distribution of the ordered set of singular values of $P,$ which is the same as the distribution of the main diagonal of $A.$ The fundamental result here is the following theorem of Y. Guivarc'h \cite{guivarc2008spectrum} (the paper is the culmination of other work of Y. Guivarc'h with his collaborators I. Goldsheid and A. Raugi, see \cite{goldsheid1996zariski,guivarc2007actions}; the first paper, together with the classic paper of Goldsheid and Margulis \cite{gol1989lyapunov} are required reading in the subject):
\begin{theorem}[\cite{guivarc2008spectrum}[Theorem 3]]
\label{guivarchthm}
If the subgroup generated by the set $S$ is \emph{Zariski dense} in $G,$ then the random variable 
\[
\frac{1}{\sqrt{n}}\left(\log A - n L(S)\right),
\] where $L(s)$ is the vector of Lyapunov exponents of the set $S,$ converges in distribution to a Gaussian random vector.
\end{theorem} 
\begin{remark}
Whether or not a subgroup of a semisimple group is Zarsiki-dense can be efficiently determined using the results of \cite{rivin2013large}.
\end{remark}
We will also use some of our own results (we are stating a special case, for the more general case, see \cite{prasad2003existence,jouve2010splitting,rivin2013large})
\begin{theorem}[\cite{rivin1999growth,rivin2010growth,rivin2008walks,rivin2009walks,MR2725508}]
\label{equidist}
Suppose that $S$ (as above) is a symmetric generating set of a Zariski dense subgroup of $\SL(n, \integers)$ or $\Sp(2n, \integers).$ Then, the products $S_n$ are equidistributed modulo $p$ for all but a finite number of primes $p.$ The set of exceptional primes is empty if $S$ generates $G(\integers).$ Furthermore, the speed of convergence is exponential in the length $n$ of the product.
\end{theorem}
and also
\begin{theorem}[\cite{rivin2008walks,rivin2009walks,MR2725508,prasad2003existence}]
\label{irredthm}
Under the assumptions of Theorem \ref{equidist}, the characteristic polynomial of a random product has probability exponentially small in $n$ of having reducible characteristic polynomial (in fact, more is true, the Galois group of the characteristic polynomial has exponentially small probability of being nongeneric -- generic being the Weyl group of the ambient Lie group (as Theorem \ref{equidist} is currently stated, the Weyl group is either $S_n$ in the case of $\SL(n),$ or $C_2 \wr S_n$ in the case of $\Sp(2n).$)
\end{theorem}
\section{Applications to homology of fibered manifolds}
Results of sections \ref{homosec} and \ref{matprods} lead, essentially immediately, to the following results.Below, a \emph{random $(\Gamma, n)$ manifold} is one whose (symplectic monodromy) is obtained by a random product of a symmetric generating set of a Zariski dense subgroup $\Gamma$ of $\Sp(2g,\integers).$
\begin{theorem}
\label{homrank}
The probability that the first Betti number of a random $(\Gamma, n)$ manifold is greater than $1$ decreases exponentially as a function of $n.$
\end{theorem}
\begin{proof} Indeed, the first Betti number exceeds $1$ only if $1$ is an eigenvalue of the monodromy, which means that, in particular, the characteristic polynomial of the monodromy is reducible over $\mathbb{Q},$ thus the result follows from Theorem \ref{irredthm}
\end{proof}
\begin{corollary}
\label{unifiber}
A generic $(\Gamma, n)$ manifold fibers in a \emph{unique} way.
\end{corollary}
See E. Hironaka's nice paper \cite{hironaka2012quotient} for a discussion of the other situation (where there are many fiberings) -- of course, there are many of those produced randomly as well, as long as we take a non-elementary subgroup of the mapping class group whose Torelli image is \emph{not} Zariski-dense in the symplectic group.
\begin{theorem}
\label{homrank2}
For a prime number $p,$ the probability of the rank of $\mathbb{F}_p$ homology of a random $(\Gamma, n)$ manifold being greater than $k$ is of the order of $1/p^k$ for $n$ large.
\end{theorem}
\begin{proof} This follows from the result attributed to A. Borel in \cite{chavdarov1995generic}, together with Theorem \ref{equidist}. The result of Borel states that the characteristic polynomials are essentially equidistributed in $\Sp(2g, \mathbb{F}_p).$ Now, it is clear that around a $1/p^k$ fraction of all characteristic polynomials have the form $(1-x)^k q(x).$
\end{proof}
\begin{theorem}
\label{torsionord}
The log of the order of torsion $T$ in $H_1(M, \integers)$ for $M$ a $(\Gamma, n)$ manifold grows linearly with $n.$ furthermore, for some quantity $a,$ (the sum of the positive Lyapunov exponents of the random matrix products corresponding to the generating set $S,$ we have $(\log T - n a)/\sqrt{n}$ is asymptotically normally distributed.
\end{theorem}
\begin{proof}
This is an immediate consequence of Guivarc'h's Theorem \ref{guivarchthm}, since the order of torsion is $T=\det(M - I),$ where $M$ is the monodromy, so for large $n$ we see that $\log T$ is basically the sum of the top $g$ singular values, which is normal by Theorem \ref{guivarchthm}.
\end{proof}
Note that Theorems \ref{homrank2} and \ref{torsionord} are in some sense opposite. The first states that it is unlikely to have more than the ``standard'' amount of $p$ torsion, for any given $p,$ the second states that the amount of torsion grows exponentially, which, together indicates that the number of prime divisors of the torsion $T$ satisfies a prime number theorem (that is, grows like $n/\log n.$)
\section{Statistics of volume and related quantities}
\label{volsec}
\subsection{Volume the hard way}
\label{volumehard}
By Brock's theorem, the volume of a surface bundle is quasi-the-same as the Weil-Petersson translation distance of the monodromy. Note that, as observed by S. Kojima in \cite{kojimanorm}, it is not necessary for the bundle to actually be hyperbolic for this statement to make sense - we define volume as simply the volume of the hyperbolic pieces in the JSJ decomposition of the mapping torus. Now, we have the following version of the multiplicative ergodic theorem (due to A. Karlsson and G. Margulis \cite{karlssonmargulis}): 
\begin{theorem}
\label{karlssonmar}
Let $G$ be a (finitely generated) semigroup of isometries of a uniformly convex, complete, CAT(0) metric space $X,$ then for almost every random product $P$ of the generators, 
\[\lim \frac1n d(y, P(y)) = A,\] and for $A> 0,$ the orbit follows a geodesic ray starting at $y.$
\end{theorem}
\begin{remark} the first part of the theorem, which is what we will be using in this section is at trivial consequence of the subadditive ergodic theorem -- the interesting part of the Karlsson-Margulis theorem \ref{karlssonmar} is the fact that random products follow geodesic rays. What is usually not trivial is showing that the magic constant $A$ is positive. This generally follows from the non-amenability of our (semi)group, and therefore we will here and in the future impose the condition that our group contains two non-commuting pseudo-Anosov elements -- this is commonly known as a ``nonelementary'' subgroup of the mapping class group.
\end{remark}
We can now state the next result:
\begin{theorem}
\label{volgrowth}
Suppose that $S=\{\gamma_1, \dotsc, \gamma_k\}$ generate a non-elementary subgroup of the mapping class group of genus $g>1.$ Then, the Gromov norm of the mapping torus of a random product $\Gamma_n$ of length $n$ of the $\gamma_i$ is almost surely contained in the interval $(nC_1, nC_2),$ for some some positive constants $C_1, C_2,$ which depend only on $S.$
\end{theorem}
In fact, since all the results apply without change to \emph{cosets} of our nonelementary group $\Gamma,$ we have the following corollary (which will be useful in the sequel):
\begin{corollary}
\label{volcor}
Let $x$ be a mapping class, and $\Gamma$ a nonelementary subgroup of the mapping class group. Then, if $\Gamma_n$ is a random product of length $n$ in $\Gamma,$ the Gromov norm of the mapping torus with monodromy $x\Gamma_n$ is almost surely contained in the interval $(nC_1, nC_2),$ where $C_1, C_2$ are the same constants as in the statement of Theorem \ref{volgrowth}.
\end{corollary}
\begin{remark}
\label{genhyp}
In fact, thanks to the results of Joseph Maher and the author (\cite{maher2011generic,rivin2008walks,rivin2009walks,MR2725508}), we know that a random element as in Theorem \ref{volgrowth} and Corollary \ref{volcor} is pseudo-Anosov, with probabiity approaching $1$ exponentially fast in $n,$ so the mapping torus is hyperbolic except in exponentially few exceptional cases. We can thus replace ``Gromov norm'' by ``hyperbolic volume''.
\end{remark} 
\begin{remark}
\label{moremaher}
The fact that the speed of divergence to infinity is non-zero also follows from Joseph Maher's theorem on speed of divergence to infinity for the curve complex distance (which is a lower bound on the pants distance and Weil-Petersson distance), see \cite{maher2010linear}
\end{remark}
\subsection{Volume the easy way}
\label{volumeeasy}
We note that the results of Namazi-Souto described in Section \ref{heegeom} apply (and are much easier) in the case of surface bundles over the circle, so the proof in that section works \emph{verbatim} here. This still uses the machine of bi-Lipschitz models, but none of the dynamics machinery used above. Of course, using the results of \cite{brocktori} in the opposite direction shows that the Weil-Petersson translation distance of random surface automorphisms has positive drift (without using the deep results of Maher).
\subsection{Rank of the fundamental group}
The fundamental group of the mapping torus of an automorphism of a surface of genus $g$ has a generating set of cardinality $2g+1,$
The following result has been claimed by Ian Biringer and Juan Souto:
\begin{theorem}[Biringer-Souto \cite{BiringerSouto}]
\label{bsthm}
Let $S$ be a surface of genus $g.$
There exists an $n_0,$ such that if $M_\phi$ be the mapping torus of a pseudo-Anosov mapping $\phi$ whose translation distance in the curve complex is greater than $n_0,$ then the rank of the fundamental group of $M_\phi$ is equal to $2g+1.$
\end{theorem}
Theorem \ref{bsthm} together with J. Maher's result \cite{maher2011generic} shows that a generic mapping torus of automorphisms in \emph{any} nonelementary subgroup $\Gamma$ (or a coset thereof) of the mapping class group has maximal rank of fundamental group. Unfortunately, while by the work of Kapovich and Weidmann it is possible to decide what the rank is, there is no remotely efficient algorithm to do so (and since, as discussed above), we can construct a coset of a nonelementary subgroup of the mapping class group in such a way as to get prescribed homology, we know that the abelianization tells us very little about the rank of the fundamental group.
\subsection{Injectivity radius}
\label{injrad}
It is known that a generic (in almost any sense) measured lamination does not have bounded complexity, and this implies that the injectivity radius of a generic mapping torus will go to zero with the length of the monodromy. What is less clear is \emph{how fast} it will go to zero. 
\begin{question}
\label{injq}
What is the expected injectivity radius (or systole length, in the punctured case) of the mapping torus of a random automorphism of length $n$ in a coset of a nonelementary subgroup $\Gamma$ of the mapping class group of a surface of genus $g.$
\end{question}
It turns out that one can get intuition for this question from the case of the punctured torus (which is not a closed surface, but this is not relevant for our purpose). This has been studied in the classical work by Yair Minsky \cite{minskypunctured}. It follows from Equation 4.4 in Minsky's paper that:
\begin{theorem}
\label{minskythm}
If the monodromy $\phi\in \SL(2, \integers)$ has the form $\begin{pmatrix}a & b\\ c& d\end{pmatrix},$ and let $[d_1, \dotsc, d_k]$ be the continued fraction expansion of $\frac{a}{b}.$ Then, the systole length of $M_\phi$ is bounded above by \[\frac{c}{(\max_{1\leq i \leq k} d_i)^2}.\]
\end{theorem}

Now, in if we look at the generating set $\tau, \sigma$ for $\SL(2, \integers,$ where
\[
\tau = \begin{pmatrix}1 & 1\\ 0 & 1\end{pmatrix}, \quad \sigma = \begin{pmatrix}0 & -1\\ 1 & 0\end{pmatrix},
\]
we see that the convergents (or ``digits'') of the continued fraction correspond to \emph{runs} of $\tau$ in the monodromy words. By the celebrated Erd\H{o}s-Renyi Theorem \cite{erdosrenyiruns} (see also \cite{reveszwalks}[Chapter 7], the length of the maximal such run in a string of length $n$ is almost surely equal to  $c \log n,$ so for that generating set we see that the injectivity radius for a random mapping torus with monodromy of length $n$ is bounded above by $c/\log^2 n,$ and we conjecture that this is true in general:
\begin{conjecture}
The injectivity radius of a random mapping torus of a surface $S$ decays as $1/\log^2 n.$
\end{conjecture}
\begin{remark}
It is clearly enough to consider loops in the ``surface'' direction. By the general Lipschitz model philosophy of Minsky, we know that a substring of the form $\phi^k$ in the monodromy looks like the cyclic cover of order $k$ of the mapping torus of $\phi,$ and thus contributes $O(k)$ to the length of the meridian loop.
\end{remark}
In fact, Minsky's methods generalize to higher genus (where our $\tau$ is now replaced by a Dehn twist around some curve), so the only question is whether there are any accidental short lo
\subsection{Bottom of the spectrum of the Laplacian}
\label{lapspec}
Consider the mapping torus $M_\phi$ of a generic element of length $n$ (with respect to your favorite coset of your favorite nonelementary subgroup of the mapping class group), and consider $\lambda_1(M_\phi).$ What can we say? On the one hand, by Schoen's result \cite{schoeneig}, for any hyperbolic manifold 
$N^3,$ we know that there exists a universal (and explicit) constant $C$ such that 
\[\lambda_1(N^3) > \frac{C}{V(N^3)}. \] Schoen's argument uses Cheeger's inequality.

On the other hand, it is a celebrated result of P. Buser \cite{busereig}, that for any such $N^3,$ we have an upper bound:
\[
\lambda_1(N^3) < 4 h(N^3) + 10 h^2(N^3),
\]
where $h(N^3)$ is the Cheeger constant of $N^3.$
Another ingredient is the following estimate due to M. Lackenby (see \cite{lackenby2006heegaard,lackenby2010finite},) based on the work of Pitts-Rubinstein \cite{pitts1987applications} which is still not entirely written up. Luckily, the requisite bound can be produced using harmonic sweepouts, as in the beautiful paper of J. Hass, A. Thompson, and W. P. Thurston \cite{hassthompthur}):
\begin{theorem}[Lackenby \cite{lackenby2006heegaard}]
\label{lackthm}
The Cheeger constant of a hyperbolic $3$-manifold $N^3$ fibered over the circle, with fiber a surface of genus $g$ is at most
\[
\frac{16 \pi g}{V(N^3}
\]
\begin{remark}
The following explicit two-sided bound were shown by D. Futer, E. Kalfagianni, and J. Purcell:
\begin{theorem}[\cite{fkp}]
Let $M^3$ be a hyperbolic manifold of volume $V,$ and Heegaard genus $g.$ Then
\[
\dfrac{\pi^2}{2^{50}V^2} \leq \lambda_1(M) \leq 32\pi \dfrac{g-1}V + 640 \pi^2 \dfrac{(g-1)^2}{V^2}.
\]
\end{theorem}
\end{remark}
\end{theorem}
Putting all of these results together with Corollary \ref{volcor}, we get:
\begin{corollary}
\label{lambdacor}
The bottom eigenvalue $\lambda_1(N)$ of the Laplacian for a mapping torus with random monodromy of length $N$ satisfies \[
\frac{C}{N^2} < \lambda_1(N) < \frac{C}{N},
\]
for some $C > 0,$ depending on the generators.
\end{corollary}
Corollary \ref{lambdacor} is obviously not entirely satisfying, and the author conjectures (following some conversations with Juan Souto):
\begin{conjecture}
With notation as above, 
\[
\lambda_1(N) \sim \frac{C}{N}
\]
\end{conjecture}
It should be noted that if we pick some $\epsilon$ and $r$ and consider hyperbolic $3$-manifolds $N^3$ of injectivity radius bounded below by $\epsilon$ and the rank of fundamental group bounded above by $r,$ then, in the beautiful paper \cite{white2013spectral}, Nina White shows: \[\frac{C_1}{V^2(N^3)} < \lambda_1(N^3) < \frac{C_2}{V^2(N^3)}.\]
In our situation, the rank is bounded, but the injectivity radius is not.
\subsection{Dilatation of a random pseudo-Anosov}
Since the log dilatation of a random pseudo-Anosov automorphism is just the top Lyapunov exponent of its derivative cocycle (this follows from the ergodicity of the action), it follows from the Oseledec ergodic theorem that the log dilatation of a random product grows linearly with the length. From this, we get the following result:
\begin{fact}
\label{dilfact} For a pseudo-Anosov mapping class $\phi \in \mcg_g,$obtained as a random word of length $N$ in the generators of a non-elementary subgroup $\Gamma$ of $\mcg_g,$ or a random element in a fixed coset $h\Gamma$ of such a subgroup, the ratio between the log of the dilatation of $\phi$ and the hyperbolic volume of $M_\phi$ lies, with probability approaching $1$ as $N \rightarrow \infty,$ in some interval $[C_1, C_2], $ where $C_1 > 0.$ If Conjecture \ref{volconj} holds, we can take $C_1 = C_2.$
\end{fact}

The relationship between dilatation and volume has received a fair amount of attention, and it is known that without a lower bound on the injectivity radius of $M_\phi,$ one has families of examples where the ratio of volume to log of the dilatation is arbitrary small (see \cite{longmorton,fathi87}) In the paper \cite{kinkojimatakasawa} the authors study this ratio and prove \emph{upper} bounds in some families of cases.

We finally remark that the dilatation comes up some quite unexpected settings, in particular that of symplectic topology, as can be gleaned from the very nice papers \cite{felshtyngrowth,smithquadrics}.
\subsection{Essential surfaces}
Every mapping torus $M_\phi(S)$ contains at least one class of essential surfaces -- those isotopic to a fiber $S.$ In the case where the rational rank of the first homology of $M_\phi(S),$ Thurston's theory of the norm on homology tells us that there are infinitely many ways to fiber $M_\phi(S),$ which gives infinitely many distinct essential surfaces. However, not all fibers are alike. There is the following result of Bachman and Schleimer (which is analogous to Kevin Hartshorn's theorem (see \cite{hartshornhee}) in the setting of Heegaard splittings with large Heegaard distance:
\begin{theorem}[D. Bachman and S. Schleimer, \cite{bachschleifib}]
\label{basthm}
If the curve complex translation distance of $\phi$ equals $N,$ then, for any essential surface $F$ not isotopic to a fiber of $M_\phi(S),$
\[
-\chi(F) \geq N.
\]
\end{theorem}
Theorem \ref{basthm} immediately implies the following corollary:
\begin{corollary}
The Euler characteristic of an essential surface not isotopic to the ``standard'' fiber in $M_\phi(S),$ for $\phi$ generated by a random walk of length $n$ in a finitely generated non-elementary subgroup $\Gamma$ of $\mcg(S)$ grows linearly in $n.$
\end{corollary}
\begin{remark}
\label{hatchflorem}
This might lead one to conjecture that such a surface might not exist, but this would be incorrect: classical results of W. Floyd and A. Hatcher \cite{flohatchfib} show that, for $S$ a punctured torus, there are \emph{always} non-fiber essential surfaces.
\end{remark}
\section{Fibered manifolds with prescribed homology}
\label{prehom}
Suppose we want a mapping torus with $b_1=1,$ and a prescribed structure of the homology group. This is easy to arrange: we find an $f\in \Sp(2g, \integers),$ with the elementary divisors of $f - I$ having the requisite structure (this is possible by Theorem \ref{snfthm}), take an $F$ whose image under the Torelli homomorphism is $f$ (that this can be done is a standard fact, see \cite{MR2850125} and multiply it by some map in the Torelli subgroup. A generic map in the Torelli subgroup is pseudo-Anosov (see \cite{malestein2013genericity,lubotzky2011sieve,maher2011generic}), and, by the work of Tiozzo \cite{tiozzo2012sublinear}, has north-south dynamics, so the same is true of a generic element in a coset (see Theorem \ref{volcor}). Furthermore, by the work of Kaimanovich and Masur \cite{kaimanovich1996poisson} on the nontriviality of the Poisson boundary of the mapping class group, we obtain the following counting result:
\begin{theorem}
\label{counttheorem}
For any finite abelian group $G,$ there exist at least  $\exp(c V),$ for some $c>0$ hyperbolic $S_g$ bundles, whose hyperbolic volume is at most $V,$ and whose first homology over the integers is isomorphic to $G \oplus \integers.$
\end{theorem}
\section{Random subgroups}
\label{randsubgp}
In this section we remark the following:
\begin{theorem}
\label{gensubgp}
Let $\gamma_1, \gamma_2$ be two random $(\Gamma, N)$ elements of a nonelementary subgroup $\Gamma$ of the  mapping class group of a surface of genus $g.$ Then with probability going to $1$ as a function of $N,$ the subgroup generated by $\gamma_1$ and $\gamma_2$ is free.
\end{theorem}
\begin{proof}
This is a standard ping-pong argument. By the work of \cite{tiozzo2012sublinear}, $\gamma_1$ and $\gamma_2$ have north-south dynamics, and by the results of Kaimanovich-Masur (which state that the harmonic measure on PML is non-atomic), we know that that for large enough $N$ the fixed points of $\gamma_1$ and $\gamma_2$ are not in each other's basin of attraction.
\end{proof}
\begin{remark} The scheme of the proof of Theorem \ref{gensubgp} is essentially the same as Guivarc'h's proof of the Tits alternative \cite{guivarc1990produits}. See also Richard Aoun's proof of a similar result for linear groups \cite{aoun2011random}.
\end{remark}
\section{Fibered cusped manifolds and Dehn filling}
\label{cusped}
In this section we look at mapping tori of automorphisms of \emph{cusped} hyperbolic surfaces. Most of the results for closed surfaces carry over essentially verbatim, so here we will focus only on results specific to cusped surfaces.

The first result we need is the nice result of D.~Futer and S.~Schleimer \cite{futer2011cusp} on the size of the cusp of mapping tori (we state here the slightly simpler version of their result where the surface has one cusp only).
\begin{theorem}
\label{futerschleimer}
Let $F$ be a punctured surface, and let $\phi: F\rightarrow F$ be an orientation-preserving pseudo-Anosov homeomorphism. Let $M_\phi$ be the mapping torus of $\phi,$ and let $C$ be the maximal cusp neighborhood in $M_\phi.$ Then, we have the following bounds:
\[
\dfrac{\overline{d(\phi)}}{450\chi(F)^4}< \area(\partial C) < 9 \chi(F)^2 \overline{d(\phi)},
\]
while
\[
\dfrac{\overline{d(\phi)}}{536\chi(F)^4} < \height(\partial C) < - 3 \chi(F)\overline{d(\phi)}.
\]
\end{theorem}
In the statement of Theorem \ref{futerschleimer}, $\overline{d(\phi)}$ refers to the stable translation length of $\phi$ in the arc complex of $F.$ As for the \emph{height} of a torus, there is a preferred direction in $\partial C,$ corresponding to the cusp of $F.$ Call the length of the geodeisc in this direction $\lambda.$ Then the height of $\partial C$ is simply \[
\height (\partial C) = \dfrac{\area(\partial C)}\lambda.\]
For our purposes, all we need to know about stable translation distance in the arc complex is that it is bounded below by a multiple of the curve complex translation distance, and bounded above by a multiple of the pants complex translation distance.
Combined with the results in Section \ref{volsec}, we get:
\begin{corollary}
\label{fscor}
Both the area and the height of the maximal cusp of a random cusped mapping torus of $F$ grow linearly with the length of the random monodromy (where the monodromy is chosen at random from a nonelementary subgroup of $\mcg(F).$)
\end{corollary}
Corollary \ref{fscor} has a somewhat surprising consequence:
\begin{theorem}
\label{dehnhyp}
Let $F$ be a surface of genus at least $2,$ with one puncture, and let $\Gamma$ be a subgroup of $\mcg(F)$ whose image under the usual ''forget the puncture'' map to $\mcg(\overline{F})$ (where $\overline{F}$ is the surface $F$ with the punctured filled in) is non-elementary, and let $\phi$ be a random element of length $n$ of $\Gamma.$ Then, with probability approaching $1$ as $n$ goes to infinity, \emph{all} Dehn fillings of $M_\phi$ are hyperbolic.
\end{theorem}
\begin{proof}
By Corollary \ref{fscor}, the height of $\partial C$ is large, and so by the argument given in \cite{agolbounds}[Theorem 8.1] (see also Theorem 8.2 in the same paper, and the similar approach in \cite{lackword}) there is at most \emph{one} exceptional slope. That slope is the trivial slope, where we just close the cusp of $F.$ However, since $\Gamma$ projects onto a non-elementary subgroup of $\mcg(\overline{F})$ under the forgetful map, the image of $\phi$ is pseudo-Anosov, with probability approaching one, so the filled-in bundle is hyperbolic also.
\end{proof}
\begin{remark}
\label{nobun}
In the case $\Gamma = \mcg(F)$ we already know that $b_1(M_\phi)$ is generically equal to $1,$ and since non-trivial Dehn filling reduces the rank of homology, all of the non-trivial fillings will be (generically) hyperbolic rational homology spheres (and, in particular, will not fiber over $S^1.$)
\end{remark}
In the case when $F$ is the punctured torus, Theorem \ref{dehnhyp} does not apply. However, we have the following result:
\begin{corollary}
\label{ptor}
In the case when $F$ is the punctured torus, generically all but the trivial filling will be hyperbolic. The trivial filling will have a $\mbox{solv}$ geometric structure.
\end{corollary}
\begin{proof}
The argument is identical to the proof of Theorem \ref{dehnhyp}, noting that a three-manifold which fibers over $S^1$ with fiber $T^2$ and \emph{hyperbolic} monodromy is a solv-manifold. See \cite{rivin2014generic}.
\end{proof}
\begin{remark}
The results of this section apply \emph{mutatis mutandis} for surfaces with multiple punctures, but stating them in that case would complicate the exposition.
\end{remark}
\subsection{Integer homology spheres}
Using Theorem \ref{snfthm} and the observations in Section \ref{prehom}, we see that we can construct a coset $H$ of the Torelli subgroup of $F$ such that for any $\phi \in H,$ $H_1(M_\phi(F), \integers) = \integers),$ and the generic such $M_\phi(F)$ is a hyperbolic manifold, and \emph{all} Dehn surgeries are hyperbolic. In particular, this is true for all integral Dehn surgeries. Since any nontrivial integral Dehn surgery on this class of manifolds will give an integer homology sphere, we see that 
\begin{theorem}
\label{intdehn}
For a generic mapping torus with monodromy in $H,$ \emph{all} nontrivial integral Dehn surgeries give \emph{hyperbolic} integral homology spheres.
\end{theorem}
By the defining property of Casson's invariant (see, for example, \cite{savcasson}), for each generic $M_\phi(F),$ the integral homology spheres described in Theorem \ref{intdehn} will be hyperbolic, and their Casson invariants will form an infinite arithmetic progression. Not only that, their volumes will be bounded above by the volume of $M_\phi(F)$ (see, for example, The step of this progression is equal to the second derivative of the Alexander polynomial (evaluated at $1$). Since the volume is (eventually) strictly increasing, almost all of these integer homology spheres are pairwise non-homeomorphic. However, it is possible that they all have the same Casson invariant, since there appears to be no easy condition which confirms that the above-mentioned $\Delta^{\prime\prime}(1)$ is not equal to zero. Luckily, in our case we can finesse this point, using the following fact:
\begin{fact}
\label{alexfact}
The Alexander polynomial of a knot $k$ in a homology sphere $\Sigma,$ where $\Sigma \backslash k$ fibers over $S^1$ is the characteristic polynomial of the Torelli action of the monodromy.
\end{fact}
Since we know that there are fibered knots in $\mathbb{S}^3$ with $\Delta^{\prime\prime}(1) = 1$ (the trefoil knots leaps to mind), we know that we can construct the requisite monodromy, and thus, just as in the beginning of this section, a non-elementary subgroup with that monodromy, and we can, therefore, conclude with:
\begin{theorem} 
\label{cassthm}
There are infinitely many mapping tori $M_\phi,$ such that the family of integral surgeries of $M_\phi$ gives a family of hyperbolic integer homology spheres whose set of Casson invariants equals $N, N+1, \dotsc
.$
\end{theorem}
\begin{remark}
A result similar to Theorem \ref{cassthm} (also with a probabilistic proof) has been announced by A.~Lubotzky, J.~Maher, and C.~Wu. Their result is stronger in that they generate \emph{all} Casson invariants (while the argument above has some difficulties with low-lying values). On the other hand, they have much less control on the spaces they construct, and their method would certainly \emph{not} even get control on volume, much less get all the examples be Dehn fillings of the same space.
\end{remark}
\subsection{Essential surfaces}
\label{esssurf}
It is an interesting, and completely open, in general, question whether a generic bundle contains any non-fiber essential surface. Luckily, thanks to the work of Floyd and Hatcher \cite{flohatchfib} and Culler, Jaco, Rubinstein \cite{cjr}, the answer is completely understood in the case where the fiber is a punctured torus. Their classification immediately implies the following counting result:
\begin{theorem}
\label{ptorcount}
Let $M_\phi$ be the mapping torus of a random pseudo-Anosov automorphism of a punctured torus, with monodromy of length $n.$ Then, the number of essential surfaces in $M_\phi$ is \emph{asymptotic} to $c^n,$ for some constant $c > 1.$
\end{theorem}
\section{Experimental results}
\label{experiments}
In this section we report on some experimental results, obtained with Mathematica and Twister. The setup was as follows: we looked at bundles of a genus 2 surface over the circle, by taking random \emph{positive} products of the Humphries generators for the mapping class group (described below in Section \ref{humphries}). The experiment consisted of taking $1000$ random products of each of the lengths $100, \dotsc, 1000,$ in increments of $10.$
\subsection{Homology}
In the first experiment we computed the natural logarithm of the order of torsion in the homology of the bundle. 
In the first figure (Fig. \ref{fig:homomean}), we plot the means of the log of the order of the torsion:

\begin{figure}
\centering
\includegraphics[width=0.5\textwidth]{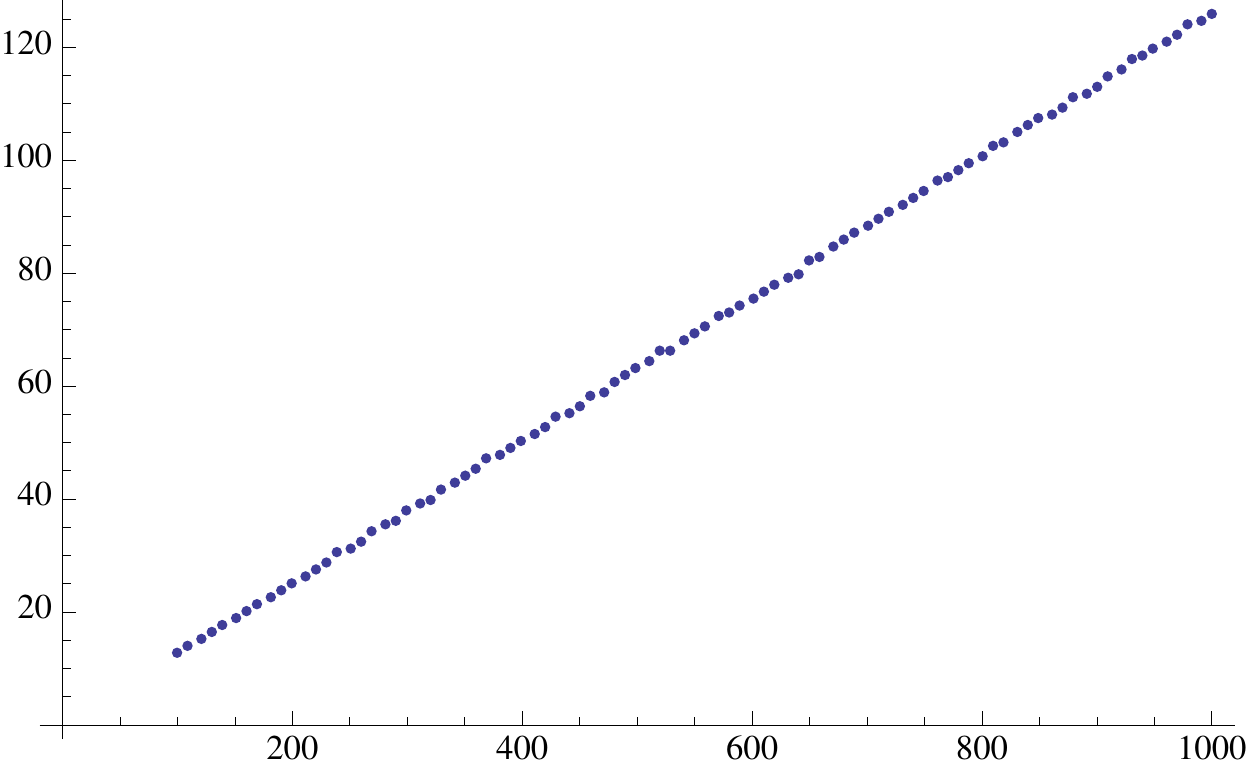}
\caption{\label{fig:homomean} Log of the order of torsion as a function of word length}
\end{figure}
We see that the experimental results bear out our theoretical results better than we could have hoped (since we have no convergence speed estimates, it would have been conceivable that we would need to go a lot further to get this sort of linear fit).
In case the reader is wondering what the slope of the straight line is, Mathematica reports that the line of best fit is:
\[
\boxed{
\mbox{$\log$ of torsion} = 0.023474 + 0.15187 \mbox{\ word length}
}
\]
Next, we plot the variance in Figure \ref{fig:homovar}.
\begin{figure}
\centering
\includegraphics[width=0.5\textwidth]{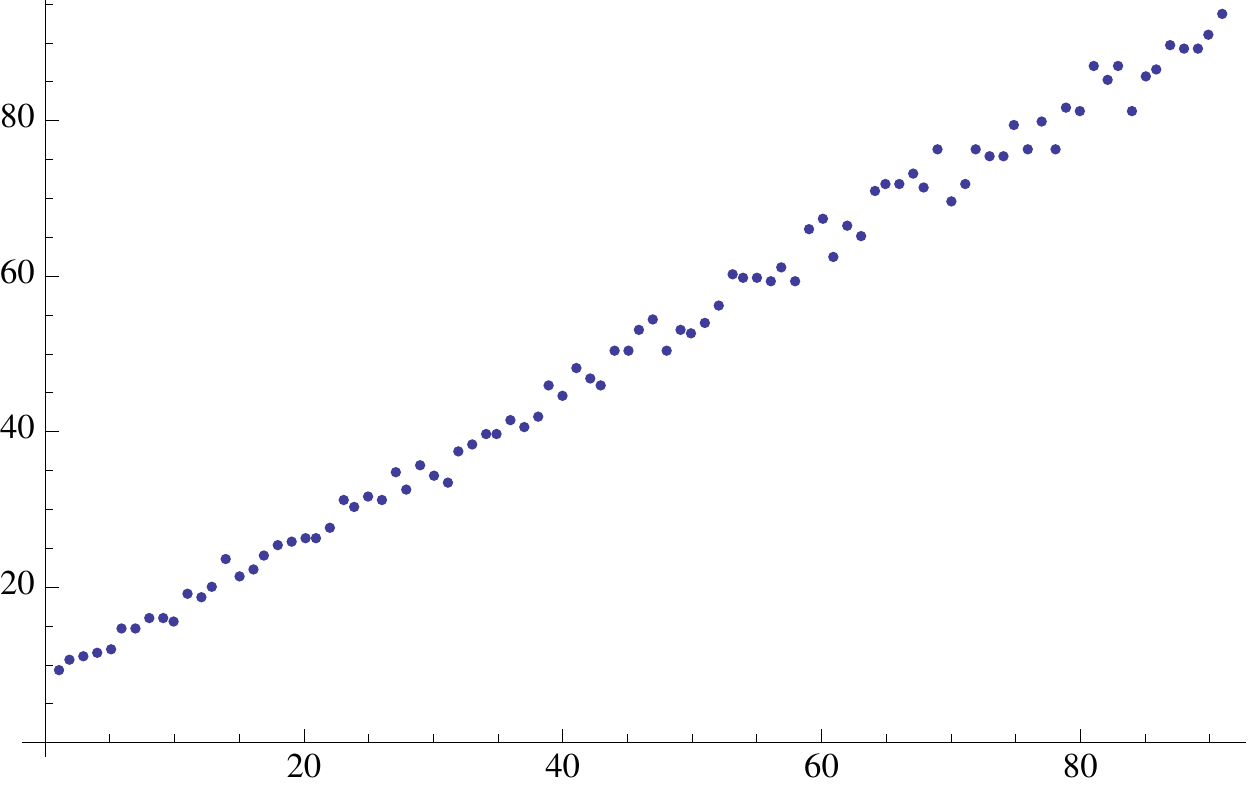}
\caption{\label{fig:homovar} The variance of the log of the order of torsion as a function of word length}
\end{figure}
We have claimed that the logs of torsion should be normally distributed for random words of the same length. Let's see if that is true.
\begin{figure}
\centering
\includegraphics[width=0.5\textwidth]{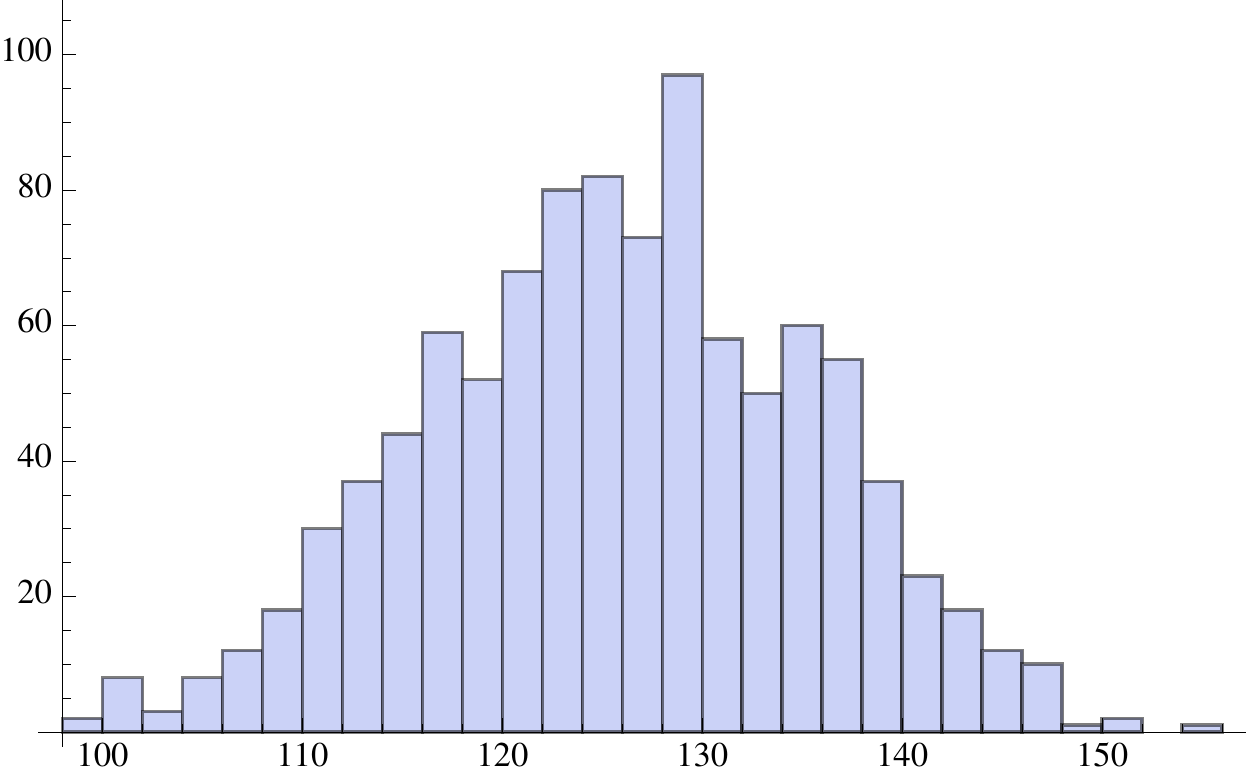}
\caption{\label{fig:homohist} The distribution of log of torsion of words of length $1000$ in the Humphries generators}
\end{figure}
As you see in Figure \ref{fig:homohist}, the distribution looks kind of normal, but to convince ourselves more, let's do a quantile-quantile plot against the normal distribution (see Figure \ref{fig:homoqq}).
\begin{figure}
\centering
\includegraphics[width=0.5\textwidth]{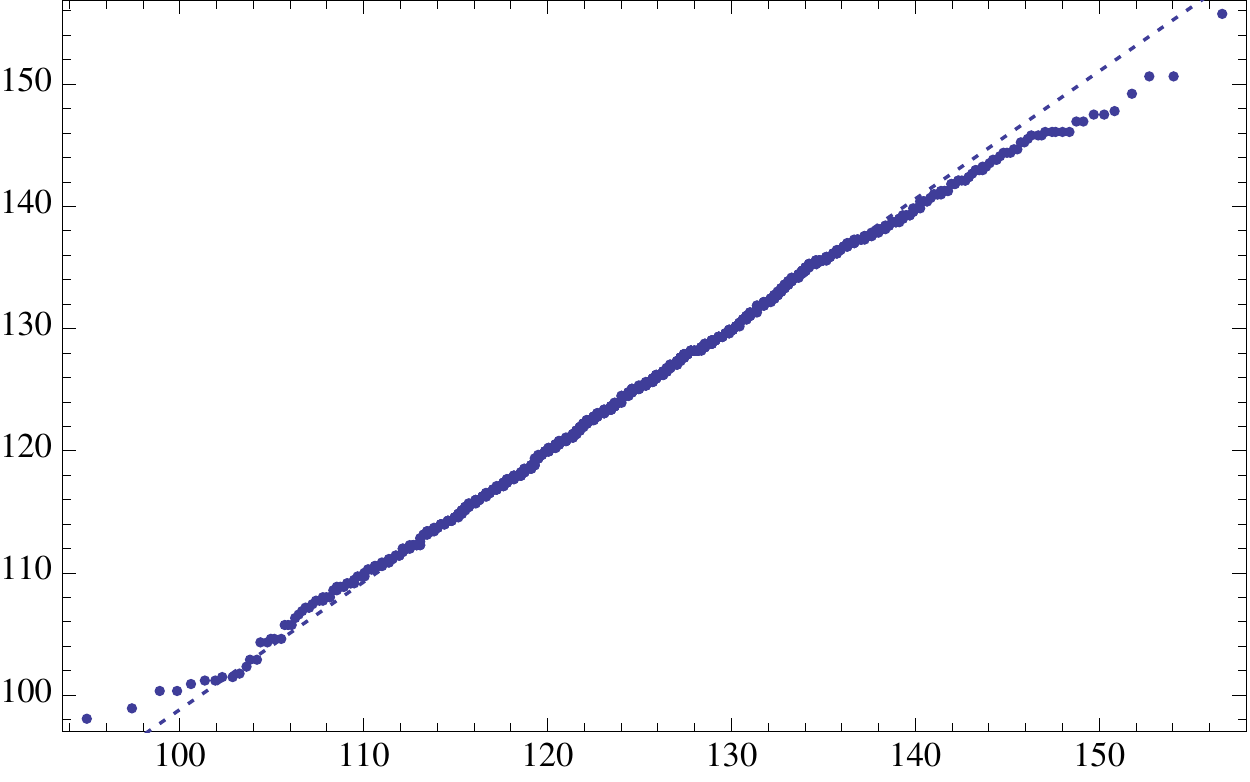}
\caption{\label{fig:homoqq} A QQ plot of the distribution of log torsion vs the standard Gaussian}
\end{figure}
We see that the fit is quite good.
\subsection{Volume}
Now, we look at the distribution of hyperbolic volumes of mapping tori of automorphisms of a surface of genus $2.$ We should remind the reader that in this setting our results are a lot weaker: we can show that the volume grows roughly linearly, but there is certainly no ``multiplicative ergodic theorem'', and, indeed, since the proofs of such are usually based on subadditivity properties of the functions considered, while hyperbolic volumes are \emph{not} subadditive\footnote{The author would like to thank Jeff Brock for alerting the author this might be true}, the proof of any such result for hyperbolic volume would require a new idea.Listing \ref{schleimerlst}, for which the author would like to thank Saul Schleimer, gives an example of non-subadditivity.
\begin{lstlisting}[caption={Non-subadditive volume},label=schleimerlst]
sage: S = snappy.twister.Surface('S_1_1')
sage: w = 'abbbbb'
sage: W = w.swapcase(); W
'ABBBBB'
sage: M = snappy.twister.Surface.bundle(S, w); M.volume()
2.02988321282
sage: M = snappy.twister.Surface.bundle(S, W); M.volume()
2.0298832128
sage: w + W
'abbbbbABBBBB'
sage: M = snappy.twister.Surface.bundle(S, w + W); M.volume()
6.0669922401
\end{lstlisting}
\begin{remark}
\label{futrem}
It was pointed out by David Futer that there is a number of ``scientific'' ways to construct families of examples.
\begin{enumerate}
\item It is a result of R.~Penner \cite{pennercurves} that if $a, b$ is a pair of curves whose union fills $F,$ then (abusing notation, so $a$ means a right-handed Dehn twist around $a$ and $b$ a right-handed twist around $b$), words of the form $a^m b^n$ gives pseudo-Anosov diffeomorphism.
\item Even simpler, since Dehn twists generate the mapping class group, products of (powers of) Dehn twists are pseudo-anosov.
\end{enumerate}
\end{remark}
With these caveats in mind, let's look at some picture (the experimental setup is the same as before: we do $1000$ experiments with each (non-reduced) word length from $100$ to $1000,$ in steps of $10.$ Firs, in Figure \ref{fig:volmean} we study the expectation of the hyperbolic volume. A single look should be enough to convince us that, submultiplicativity or no, there is a result there. A regression give the following empirical relationship:
\[
\boxed{\mbox{expected hyperbolic volume} = -0.105385 + 0.665259 \mbox{\ word length}}
\]
\begin{figure}
\centering
\includegraphics[width=0.5\textwidth]{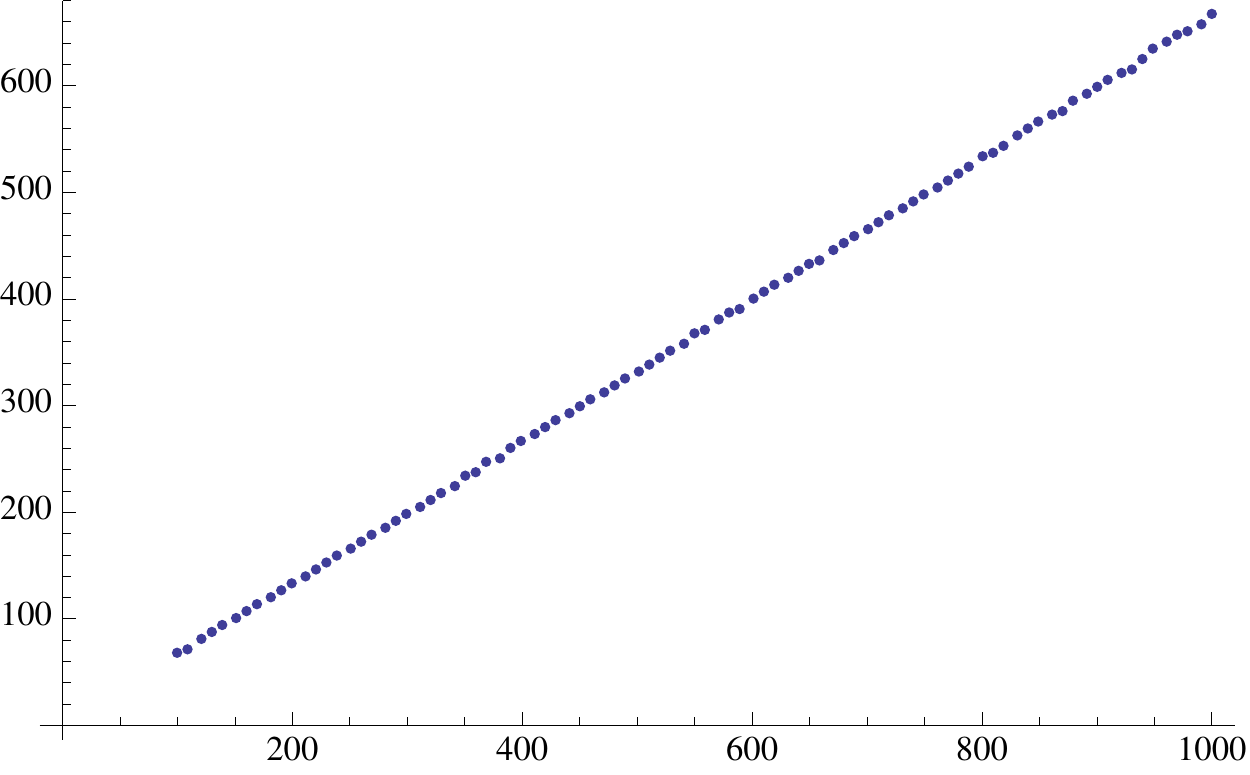}
\caption{\label{fig:volmean} The mean volume of random monodromies in the mapping class group of the surface of genus $2.$}
\end{figure}
Let's now look at the variances (to see if we should be ambitious, and try to prove a central limit theorem, as well). This is the content of Figure \ref{fig:volvar}, which shows that a central limit theorem may well be in the cards.
\begin{figure}
\centering
\includegraphics[width=0.5\textwidth]{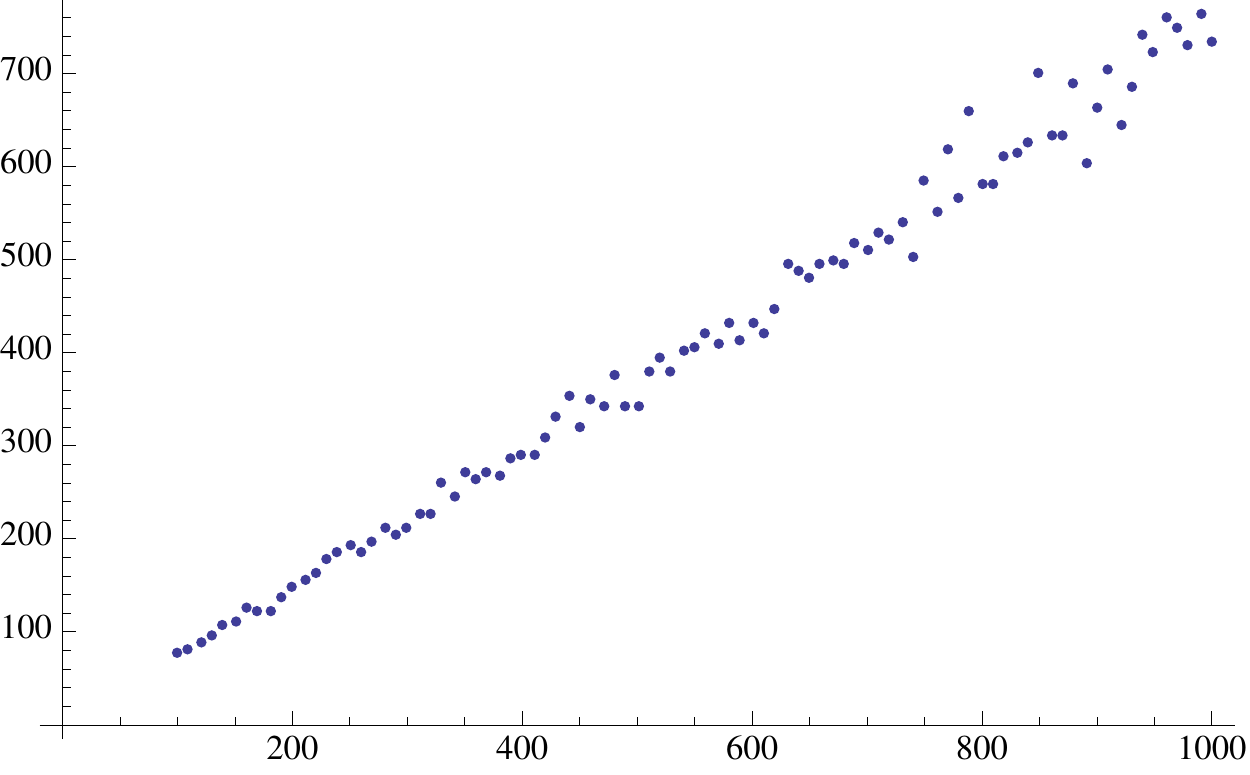}
\caption{\label{fig:volvar} The variance of volumes as a function of word length}
\end{figure}
Finally, let's look at the distribution of volumes for a fixed word length. First, a histogram (Figure \ref{fig:volhist}), which should bolster our confidence, and then a quantile-quantile plot (Figure \ref{fig:volqq}).
\begin{figure}
\centering
\includegraphics[width=0.5\textwidth]{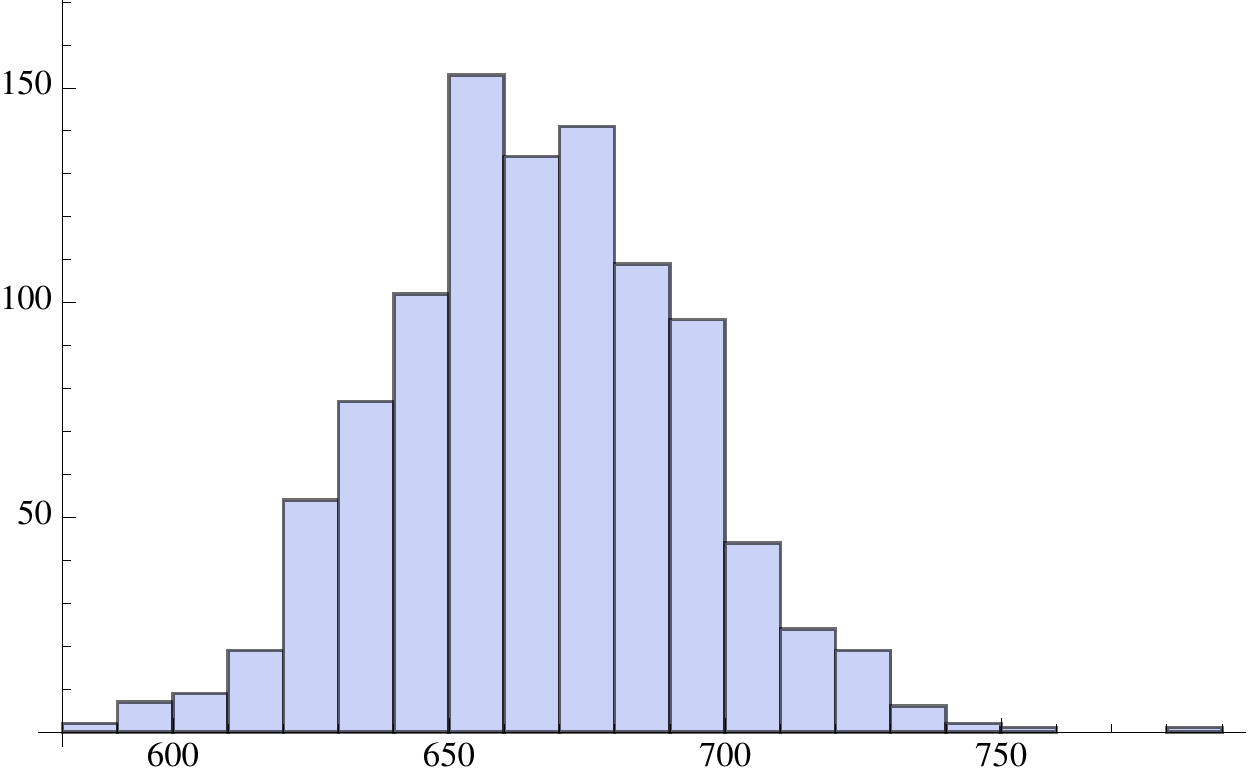}
\caption{\label{fig:volhist} Distribution of volumes of unreduced words of length $1000$}
\end{figure}
\begin{figure}
\centering
\includegraphics[width=0.5\textwidth]{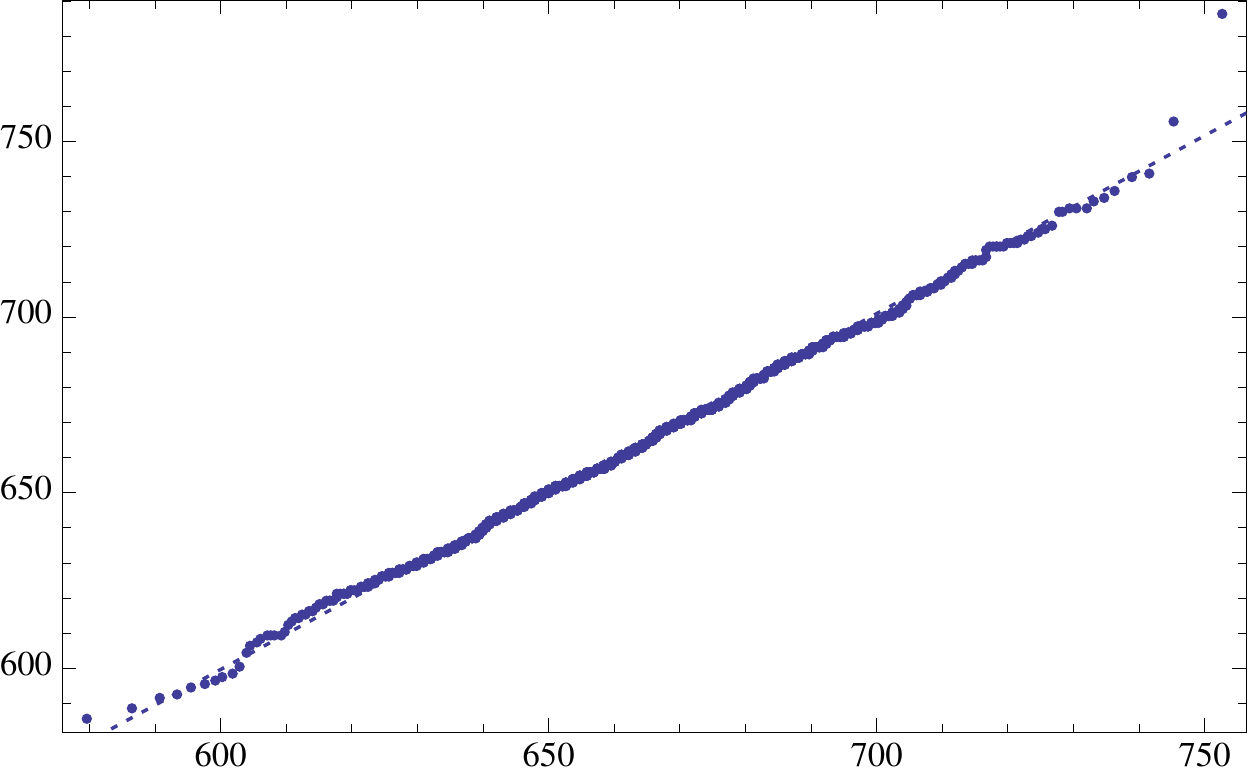}
\caption{\label{fig:volqq} QQ plot of volumes of unreduced words of length $1000$ vs the standard Gaussian}
\end{figure}
The results seem to indicate that there is an ergodic \emph{and} a central limit theorem to be discovered. We state this as a conjecture:
\begin{conjecture}
\label{volconj}
For any fixed generating set $\Gamma$ of the mapping class group of the surface of genus $g,$ let $\phi_N$ be given by a random word of length $N.$ Then, if $V(\phi_N)$ is the volume of the mapping torus of $\phi_N,$ then there exists a constant $c(\Gamma)$ such that 
$V(\phi_N)/N$ converges almost surely to $c(\Gamma).$ Furthermore, the quantity $\frac{1}{\sqrt{N}}(V(\phi_N) - N c(\Gamma))$ converges in distribution to $N(0, \sigma^2),$ for some positive $\sigma = \sigma(\Gamma).$
\end{conjecture}
\subsection{Volume and homology}
Wolfgang L\"uck in \cite{luckbook} and N. Bergeron and A. Venkatesh in  \cite{bergeronasymptotic} conjecture that for towers of congruence covers, the ratio of log of the torsion in homology divided by the volume approaches a limit. We see that (at least empirically) a similar phenomenon holds for random mapping tori with "complicated" monodromy. In particular, for such a mapping torus, the ratio of volume to log of the size of the torsion in the first homology group approaches around $4.381$ for genus 2 bundles. In Table \ref{tab:volhom} below, we compute the growth rates for volume and log of torsion as a function of growth rates. It seems (though is not completely clear) that both the growth rates approach a limit as genus goes to infinity, and so, consequently, does their ratio. We will boldly state this as a conjecture:
\begin{conjecture}
\label{limconj}
The (expected) growth rates of volume and logarithm of torsion as a function of word length of monodromy approach both approach a finite limit as genus of the fiber goes to infinity.
\end{conjecture}
\begin{remark}
In the setting of L\"uck''s, and Bergeron-Venkatesh \cite{luckbook,bergeronasymptotic} the ratio of the volume to logarithm of the torsion should approach $6\pi,$ which the numbers in the third column of Table \ref{tab:volhom} clearly do not. The boldest possible conjecture is that they approach $3\pi.$ The reader should not be perturbed by the difference of the numbers, since our model is really quite different\footnote{Stefan Friedl tells the author that there are good reasons to believe that the answer should, indeed, be $3\pi,$ though the evidence of Table \ref{tab:volhom} is certainly not definitive.}
\end{remark}
\begin{table}
\centering
\begin{tabular}{|c|c|c|c|}
genus & log torsion growth & volume growth & ratio \\\hline
2 & 0.15187 & 0.665295 & 4.38068 \\
3 & 0.163482 & 1.04903 & 6.41679 \\
4 & 0.16938 & 1.18395 & 6.9899 \\
5 & 0.172869 & 1.26366 & 7.30992 \\
6 & 0.175353 & 1.31758 &7.51387 \\
7 & 0.177159 & 1.35517 &7.649456 \\
8 & 0.178507 & 1.38492 & 7.75835 \\
9 & 0.179699 & 1.40891 & 7.84038 \\
10 & 0.180552 & 1.42325 & 7.88277 \\
11 & 0.181283 & 1.44077 & 7.94763 \\
12 & 0.181848 & 1.45311 & 7.99079 \\
13 & 0.182397 &  1.46369 & 8.02475\\
14 & 0.182987 & 1.47036 & 8.03532
\end{tabular}
\caption{\label{tab:volhom}Growth rates of log of torsion and hyperbolic volume}
\end{table}

\subsection{The Humphries generators for the Mapping Class Group of a closed surface}
\label{humphries}
Steve Humphries showed in \cite{humphries1979generators}that one has the following set of generators for the mapping class group of a surface of genus $g.$
First, we use the notation in Figure \ref{fig:humph}
\begin{figure}
\centering
\includegraphics[width=0.5\textwidth]{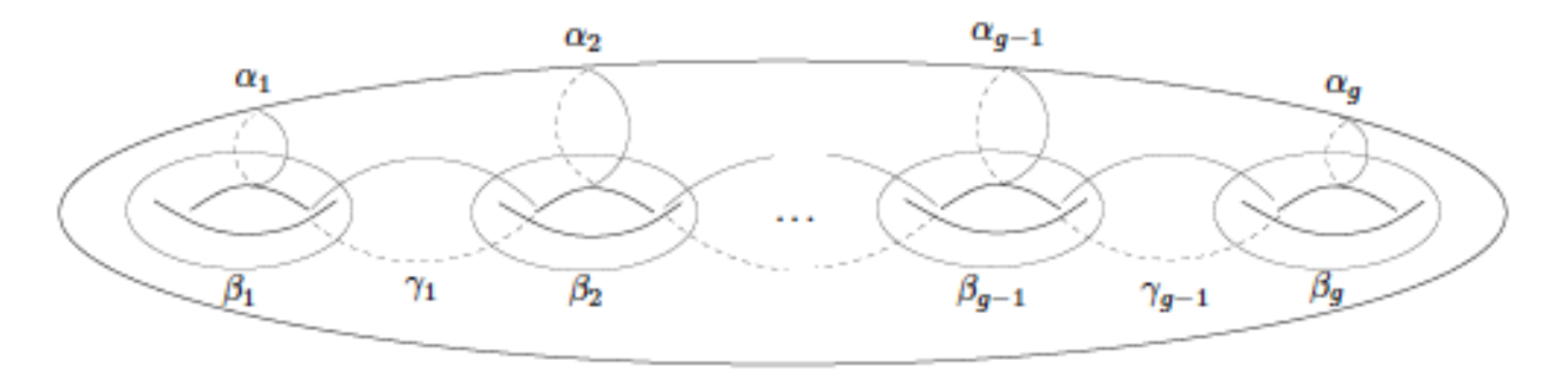}
\caption{\label{fig:humph} defining the Humphries generators}
\end{figure}
With the notation introduced in Figure \ref{fig:humph}, the Humphries generating set is the set of Dehn twists $\beta_1, \dotsc, \beta_g, \gamma_1, \dotsc, \gamma_{g-1}, \alpha_1, \alpha_2.$
With a little bit of work, we can figure out the action of the Humphries generators on the first homology of the surface (Joan Birman had previously showed  (see \cite{birman1971siegel}) that these matrices generated $\Sp(2g, \integers)$). We give the corresponding matrices below in Mathematica code:
\begin{lstlisting}
Ematrix[i_, j_, n_] := 
 Table[If[(k == i && j == l) || k == l, 1, 0], {k, 1, n}, {l, 1, n}]

BirmanY[g_, i_] := 
 Table[If[k == l, 1, If[k == i && l == g + i, -1, 0]], {k, 1, 
   2 g}, {l, 1, 2 g}]

BirmanU[g_, i_] := 
 Table[If[k == l, 1, If[l == i && k == g + i, 1, 0]], {k, 1, 
   2 g}, {l, 1, 2 g}]

BirmanZ[g_, i_] := 
 Table[If[k == l, 1, 
   If[k == i && l == g + i, -1, 
    If[k == i + 1 && l == g + i + 1, -1, 
     If[k == i && l == g + i + 1, 1, 
      If[k == i + 1 && l == g + i, 1, 0]]]]], {k, 1, 2 g}, {l, 1, 
   2 g}]

BirmanHump[g_] := 
 Union[Table[BirmanU[g, i], {i, 1, g}], 
  Table[BirmanZ[g, i], {i, 1, g - 1}], {BirmanY[g, 1], BirmanY[g, 2]}]
\end{lstlisting}
The command \lstinline$BirmanHump[g]$ produces the full list of images of the Humphries generators.
\section{Other generating sets}
The fact that the growth rate of torsion seems to be asymptotically independent of genus would seem to be more a statement about linear groups than about surfaces, and it would seem useful to conduct the experiment with other natural generating sets for the symplectic group, as well as for the special linear group (where the more natural quantity to study is the sum of the positive Lyapunov exponents). We conduct the study with the small generating sets for both groups: for the special linear group, the natural generating set is the one of cardinality $2$ discovered by Hua and Reiner. For the symplectic group, the natural generating set is the one discovered by Stanek. To keep the presentation consistent, we give Mathematica code to generate both generating sets, and then give our experimental results in Tables \ref{tab:sln} and \ref{tab:sp2n}.The results in those two tables are for the random products of \emph{positive} powers of the generators.

First, the Hua and Reiner (\cite{hua1949generators} generators:
\begin{lstlisting}
HRU2[n_] := 
 Table[If[i == j || (i == 1 && j == 2), 1, 0], {i, 1, n}, {j, 1, n}]
 
 HRU5[n_] := 
 Table[If[i - j == 1, 1, If[i == 1 && j == n, (-1)^(n - 1), 0]], {i, 
   1, n}, {j, 1, n}]
   
 hrgenlist[n_] := {HRU2[n], HRU5[n]}
\end{lstlisting}
The last command \lstinline$hrgenlist[n]$ produces the pair of Hua-Reiner generators for the special linear group.

Now, the Stanek generators (see \cite{stanek1963two}):
\begin{lstlisting}
R21[n_] := Table[
  If[i == j, 1, 
   If[i == 2 && j == 1, 1, If[i == n + 1 && j == n + 2, -1, 0]]], {i, 
   1, 2 n}, {j, 1, 2 n}]
      
Tk[n_, k_] := 
 Table[If[i == j, 1, If[i == n + k && j == k, 1, 0]], {i, 1, 
   2 n}, {j, 1, 2 n}]
   
Dd[n_] := Table[
  If[j - i == 1 && (i < n || (i >= n + 1 && i < 2 n)), 1, 
   If[i == n && j == n + 1 , -1, If[ i == 2 n && j == 1, 1, 0]]], {i, 
   1, 2 n}, {j, 1, 2 n}]
   
stanekgens[n_] := 
  If[n == 2 || n == 3, {R21[n], Tk[n, 1], Dd[n]}, 
   If[n == 1, hrgenlist[2], {R21[n] . Tk[n, 1], Dd[n]}]]
\end{lstlisting}

The command \lstinline$stanekgens[n]$ produces the Stanek generators of the symplectic group. for $\Sp(2, \integers)$ and $\Sp(3, \integers)$ the generating set has cardinality $3,$ and otherwise cardinality $2.$

Interestingly, the sum of the to $n$ Lyapunov exponents of the Stanek generators seems to quickly increase to the asymptote in the case of $\Sp(2n, \integers),$ while the sum of the positive Lyapunov exponents is \emph{decreasing} (not obviously to any asymptotic lower bound) for the Hua-Reiner generators.
\begin{table}
\centering
\begin{tabular}{|c|c|}
$n$ & sum of the positive Lyapunov exponents \\\hline
2 & 0.00108941 \\
3 & 0.00108922 \\
4 & 0.00108562 \\
5 & 0.00107511 \\
6 & 0.0010696 \\
7 & 0.00107539 \\
8 & 0.00106576 \\
9 & 0.00106297 \\
10 & 0.00106195 \\
11 & 0.00105847 \\
12 & 0.00105281 \\
13 & 0.00104964 \\
14 & 0.00104896 \\
15 & 0.00104435
\end{tabular}
\caption{\label{tab:sln}The sum of the positive Lyapunov exponents of the Hua-Reiner semigroup for $\SL(n, \integers)$}
\end{table}
\begin{table}
\centering
\begin{tabular}{|c|c|}
$n$ & $\log \det (M - I)$ \\\hline
2 & 0.132424 \\
3 & 0.140012 \\
4 & 0.0.191286 \\
5 & 0.192716 \\
6 & 0.193323 \\
7 & 0.19334 \\
8 & 0.193514 \\
9 & 0.193462 \\
10 & 0.193432 \\
11 & 0.193047 \\
12 & 0.193162 \\
13 & 0.193178 \\
14 & 0.19321 \\
15 & 0.0.193219
\end{tabular}
\caption{\label{tab:sp2n}The sum of expectation of the positive Lyapunov exponents of the Stanek semigroup for $\Sp(2n, \integers)$}
\end{table}
\section{Random Heegard Splittings}
\label{randomh}
As mentioned in the Introduction, this work has its genesis in the foundational paper of N. Dunfield and W.~P.~Thurston\cite{dunfield2006finite}. Let us recall Dunfield and Thurston's setup. We first pick an integer $g>1.$ Then, we take two handlebodies $H_1$ and $H_2$ of genus $g,$ and identify the boundary of $H_1$ to that of $H_2$ via a map $\phi$ -- it is well-known that the topology of the resulting $3$-manifold  $H_\phi^3$ depends only on the isotopy class of $\phi,$ so we can produce random manifolds by generating random elements $\phi$ from $\mcg_g,$ which is, indeed, what Dunfield and Thurston do. They then study a number of questions (they work in much lesser generality than we do in this paper, and their questions are for the \emph{whole} mapping class group. Below, we will try to indicate the appropriate level of generality.
\subsection{Homology} To compute the homology of $H_\phi^3,$ Dunfield and Thurston note that if $\iota_* : H_1(\partial H_g, \integers) \rightarrow H_1(H_g, \integers)$ is the natural map on homology induced by the inclusion map of the surface $\partial H_g$ into the handlebody, then, if the two copies of $H_g$ are identified by the map $\phi$ to produce a three-manifold $H_\phi^3,$ then the homology of $H_\phi^3$ is the quotient of $H_1(\partial H_g, \integers)$ by the subspace generated by $J=\ker \iota_*$ and $\phi_*^{-1}(J).$ They note that, with respect to the intersection pairing on $H_1(\partial H_g, \integers),$ the subspace $J$ is Lagrangian (since $J$ is generated by the $g$ meridians, which are pairwise disjoint). Dunfield and Thurston then use this observation and some hand-counting of Lagrangian subspaces to compute the asymptotic behavior of $H_1(H_\phi^3, \mathbb{F}_p),$ and thus deduce that the first Betti number of $H_\phi^3$ is generically equal to zero (so that $H_\phi^3$ is generically a rational homology sphere), while it is generically \emph{not} an integer homology sphere. Here we note that, by using somewhat more sophisticated methods we can get a much stronger result:
\begin{theorem}
\label{expgrowthDT}
If the gluing map $\phi$ is chosen uniformly at random from words of length $N$ in some generating set of $\mcg(g),$ the expected cardinality of the logarithm of the cardinality of the first homology group of $H_\phi^3$ is linear in $n.$ Furthermore, the logarithm satisfies a central limit theorem:
\[
\frac{1}{\sqrt{n}} \left(\log |H_1(H_\phi^3)| - n \lambda\right) \sim N(0, \sigma^2),
\]
for some positive constants $\lambda, \sigma.$ Furthermore the probability that $\left|\log |H_1(H_\phi^3)| - n \lambda\right| > m \lambda$ decays exponentially with $m.$
\end{theorem}
\begin{remark} In the proof we only use the fact that $\Sp(2g, \integers)$ -- the image of $\mcg(g)$ under the Torelli map -- is Zariski-dense in $\Sp(2g),$ so our results apply without change to any subgroup of $\mcg(g)$ which has Zariski-dense Torelli image.
\end{remark}
\begin{proof}[Proof of Theorem \ref{expgrowthDT}]
To prove Theorem \ref{expgrowthDT} let's first think $J$ as the "standard" Lagrangian subspace of $\mathbb{R}^{2g},$ and let $K$ be the standard complementary Lagrangian subspace. Consider a symplectic map $\Sigma.$ In the basis $(J, K),$ we have
\[
\Sigma= \begin{pmatrix} A(\Sigma) && B(\Sigma) \\ C(\Sigma) && D(\Sigma) \end{pmatrix}.
\]
The Dunfield-Thurston description tells us that $H_1(H_\phi^3) \simeq \coker B(\phi_*),$ so this is the object we will study in the sequel.

Consider the action of $\Sp(2g, \reals)$ on the space of $g$-isotropic vectors of $\bigwedge^g \reals^{2g}$ (this is the span of basis of Lagrangian $g$-dimensional spaces in $\bigwedge^g \reals^{2g}.$ This action is well-known to be irreducible (see, e.g. \cite{procesilie,goodmanwallach}). Since $\Sp(2g, \reals)$ is an $\reals$-split algebraic group, the results of Abels, Margulis, and Soifer \cite{AbelsMargulisSoifer}[Section 6] together with the results of Goldsheid and Margulis \cite{GoldsheidMargulis89} show that any Zariski-dense subgroup contains a proximal element, and so its action is contracting. Since $\Sp(2g, \reals)$ is Zariski-connected, it follows that the action of any Zariski-dense subgroup on the $g$-isotropic subspace is \emph{strongly irreducible}. These to facts together imply that we can use the theory of random matrix products, specifically the Central Limit Theorem and the Large Deviation Theorem for the matrix coefficients of such random products, as found, for example, in \cite{bougerol1985products}[Chapter V], upon remarking that the determinant of $V(\Sigma)$ is just the $JK$th matrix coefficient of the induced action of $\Sigma$ on the $g$-isotropic subspace of $\bigwedge^g \reals^{2g}.$
\end{proof}

An amusing corollary of Theorem \ref{expgrowthDT} is the following:
\begin{corollary}
\label{matcor}
The complexity of $H_\phi^3$ with $\phi$ chosen as in the statement of Theorem \ref{expgrowthDT} grows linearly with the length $n$ of the walk.
\end{corollary}
Recall that the \emph{complexity} of a hyperbolic $3$-manifold $M^3$ is defined as the minimal number of simplices in a (semi-simplicial) triangulation of $M.$
\begin{proof}[Proof of Corollary \ref{matcor}]
The linear upper bound is immediate. For the lower bound, 
we use the result of Matveev and Pervova \cite{matper}, which states that the complexity of $M^3$ is bounded below by the logarithm base $5$ of the torsion in the first homology of $M^3$ -- see \cite{matalgo}[Theorem 2.6.2]
\end{proof}
\subsection{Betti numbers of covering spaces}
Dunfield and Thurston \cite{dunfield2006finite} raise the following question:
\begin{question} Suppose $N_\phi$ is a random Heegard splitting (in the usual sense), and suppose $\tilde{N_\phi}$ has a Galois covering space with deck group $Q.$ What is the probability that $\tilde{N_\phi}$ has first nonzero first Betti number?
\end{question}
Dunfield and Thurston show that the probability tends to $0$ for a random gluing map $\phi$ under the (obviously restrictive) conditions that the genus of the handlebody equals $2$ and the deck group $Q$ is abelian -- their proof is an elaboration of their argument that a random Heegard splitting is a homology sphere, and is an explicit bare-hands argument.

A refinement of the proof of Theorem \ref{expgrowthDT} gives the result in much greater generality:
\begin{theorem}
\label{expgrowthDT2}
Let $Q$ be a finite \emph{solvable} group. Then the probability that the $3$-manifold $N_\phi$ obtained from a random Heegard splitting of genus $g > 1$ has $b_1 > 0$ tends to $0$ (exponentially with the length of the word defining the gluing map).
\end{theorem}
\begin{proof}[Proof sketch]
The rational homology of the $Q$ cover $\tilde{N_\phi}$ of  $N_\phi$ splits into irreducible representations of $Q,$ and the proof of Theorem \ref{expgrowthDT2} would follow immediately by the same argument as that of Theorem \ref{expgrowthDT} \emph{if} we knew that the action of the mapping class group on each irreducible were a Zariski dense subgroup of the symplectic group. An adaptation of the argument of \cite{grunewald2013arithmetic} (where the setup is slightly different) shows that the action is \emph{arithmetic} (that is, a lattice in the corresponding symplectic group), under a technical condition ("$\phi$-redundancy") which is satisfied when $Q$ is solvable.
\end{proof}
\subsection{Other topological invariants}
In their foundational paper \cite{namazisoutoheegard} H. Namazi and J. Souto show that a manifold $M_{f^n}^3$ obtained by gluing two handlebody of genus $g$ by a (sufficiently high) power of a ``generic pseudo-Anosov'' has Heegaard genus $g$ and rank of fundamental group equal to $g.$ Since a random map $\phi$ given by a sufficiently long word in (some) generators of $\mcg(g)$ has arbitrary high power of any given element $f$ as a subword, we have 
\begin{theorem}
\label{genusrank} manifold obtained by a random Heegaard genus of genus $g$ by a random gluing map has both Heegaard and rank of the fundamental group equal to $g.$
\end{theorem}
Another natural question (closely related to Dunfield and Thurston's set of questions) is whether a random Heegaard splitting is Haken. Unfortunately, this seems to be a hard problem, but what we can say is that if an incompressible surface exists, it is of a very high genus. This follows from two facts. The first, is the following theorem of Kevin Hartshorn \cite{hartshornhee}:
\begin{theorem}[K. Hartshorn]
Let $M$ be a Haken $3$-manifold containing an imcompressible surface of genus $g.$ Than \emph{any} Heegaard splitting of $M$ has translation distance \emph{at most} $2g.$
\end{theorem}
The second is J. Maher's \cite{maher2010linear} result that the translation distance of a random automorphism $\phi$ increases linearly with the combinatorial length of $\phi.$ Put together, these imply:
\begin{theorem}
\label{nonhakthm}
For a random gluing map $\phi$ of combinatorial length $n,$ $N_\phi^3$ almost surely has no incompressible surfaces of genus smaller $c n,$ where the constant $c$ depends on the genus and the generating set.
\end{theorem}

Theorem \ref{nonhakthm} invites the obvious conjecture:
\begin{conjecture}
\label{nonhakconj}
The manifold $N_\phi^3,$ as in the statement of Theorem \ref{nonhakthm}, is almost surely not Haken.
\end{conjecture}
\subsection{Geometric quantities}
\label{heegeom}
The central geometric quantity is clearly the hyperbolic volume. Dunfield and Thurston conjecture that volume grows \emph{linearly} in the length of the gluing map. The results of H. Namazi and J. Souto \cite{namazisoutoheegard} tell us since a random gluing map has ``blocks'' of the form $f^k$ for some (actually, every possible short enough) $f,$ and each such block looks like a cover of the mapping torus of $f,$ we have the following theorem, which resolves the conjecture of Dunfield-Thurston:
\begin{theorem}
\label{geomheeg}
The volume of $N_\phi^3,$ for $\phi$ a random word of length $n$ grows linearly in $n.$ Furthermore, the Cheeger constant of $N_\phi^3$ goes to zero, as does the bottom eigenvalue of the Laplacian, as does the injectivity radius. The rates of change are the same as for fibered surfaces.
\end{theorem}
\begin{proof}
The statement on the volume follows from the discussion above, and the trivial observation that the word giving $\phi$ will have a linear number of  blocks of the form $f^k$ for \emph{any} $f$ of combinatorial length $O(1).$ Indeed, one can (but we don't need to) show sharp limit theorems -- see \cite{FellerV1}[XIII.7,eq.~(7.8)].

The statements on the other geometric quantities follow from the results of \cite{namazisoutoheegard} and the arguments for fibered manifolds.
\end{proof}
\begin{remark}The astute reader will note that we only show a linear \emph{lower} bound. However, an upper bound follows from a standard argument which shows that the number of simplices in a triangulation of $N_\phi^3$ grows at most linearly in the length of $\phi,$ hence so does the hyperbolic volume.
\end{remark}
\subsection{Experimental Results}
The proof of Theorem \ref{geomheeg} only shows that volume grows ``coarsely linearly'', but of course it is natural to ask what the truth is. The author has conducted extensive experiments, and below are two sample graphs (for genus equal $2$ and $5.$) The experiment consisted of computing one thousand random splitting of each of the lengths consider (roughly from $200$ to $1000$ in multiples of $10$), and plotting the mean. The reader can see that the graphs are about as linear as one could wish for, which leads us to strengthen the Dunfield-Thurston conjecture:
\begin{conjecture}
\label{DTvol2}
Volume of random Heegaard splittings of combinatorial length $n$ satisfy a central limit theorem: there exist positive constants $\beta$ and $\sigma$ such that 
\[
\frac{1}{\sqrt{n}}(V(N_\phi^3) - n \beta)) \rightarrow N(0, \sigma^2).
\]
\end{conjecture}
To convince the reader that the conjecture is plausible we include also the graph of the variances for each length, and the histograms of the empirical distributions.
\begin{figure}
\centering
\includegraphics[width=0.5\textwidth]{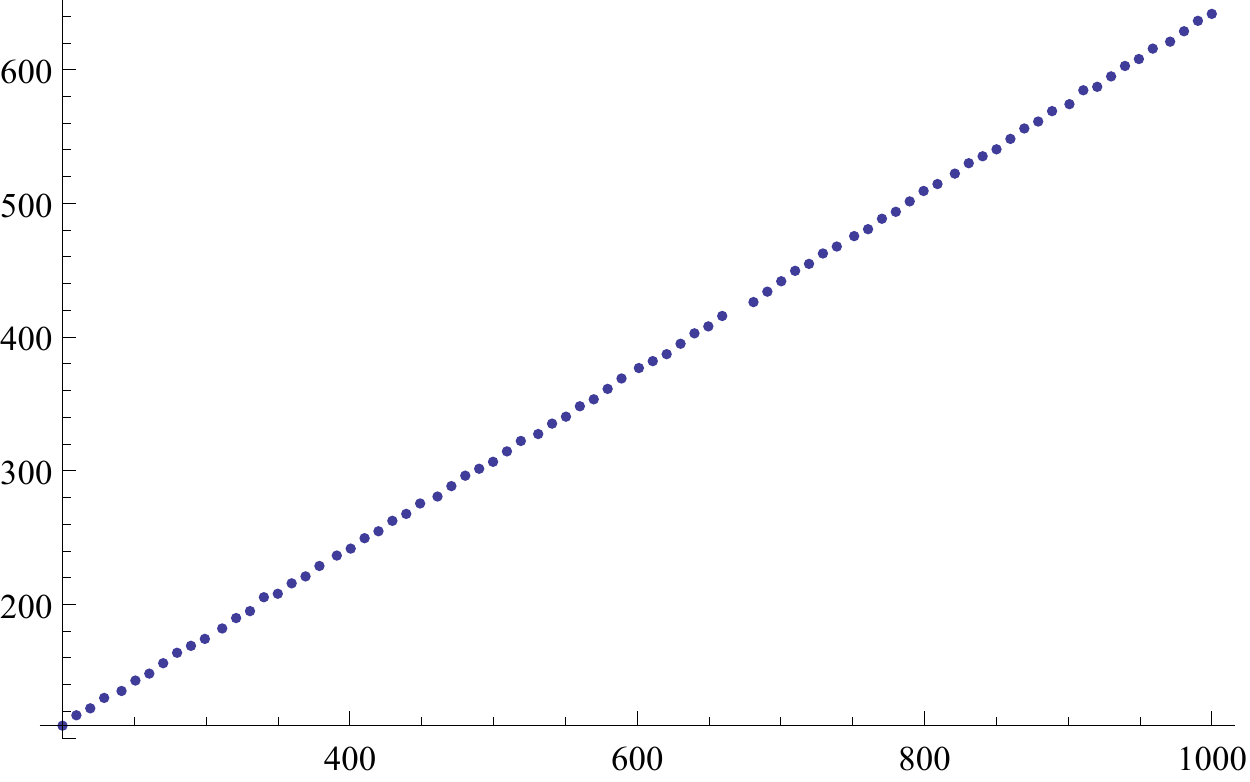}
\caption{\label{fig:graphg2} Mean values of volume for random splittings of genus 2 \\(slope $0.66$)}
\end{figure}
\begin{figure}
\centering
\includegraphics[width=0.5\textwidth]{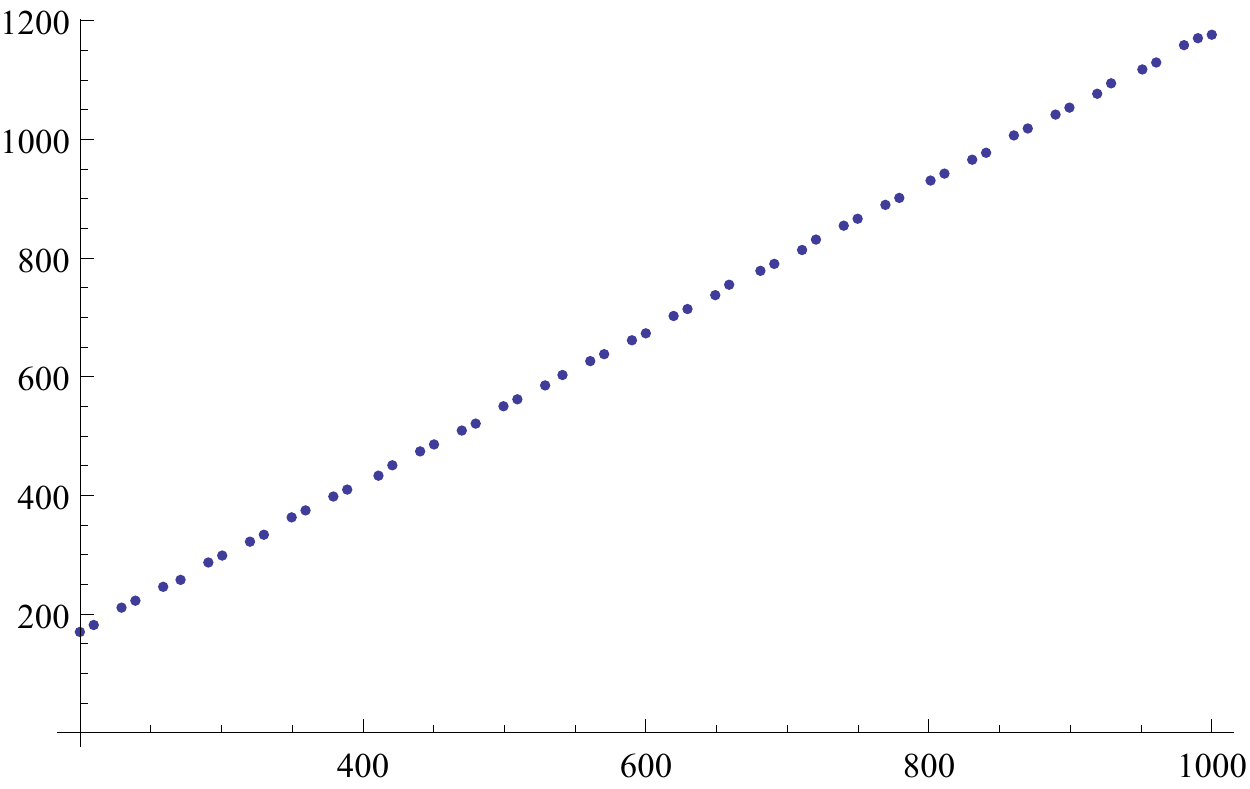}
\caption{\label{fig:graphg5} Mean values of volume for random splittings of genus 5 \\(slope $1.26$)}
\end{figure}
\begin{figure}
\centering
\includegraphics[width=0.5\textwidth]{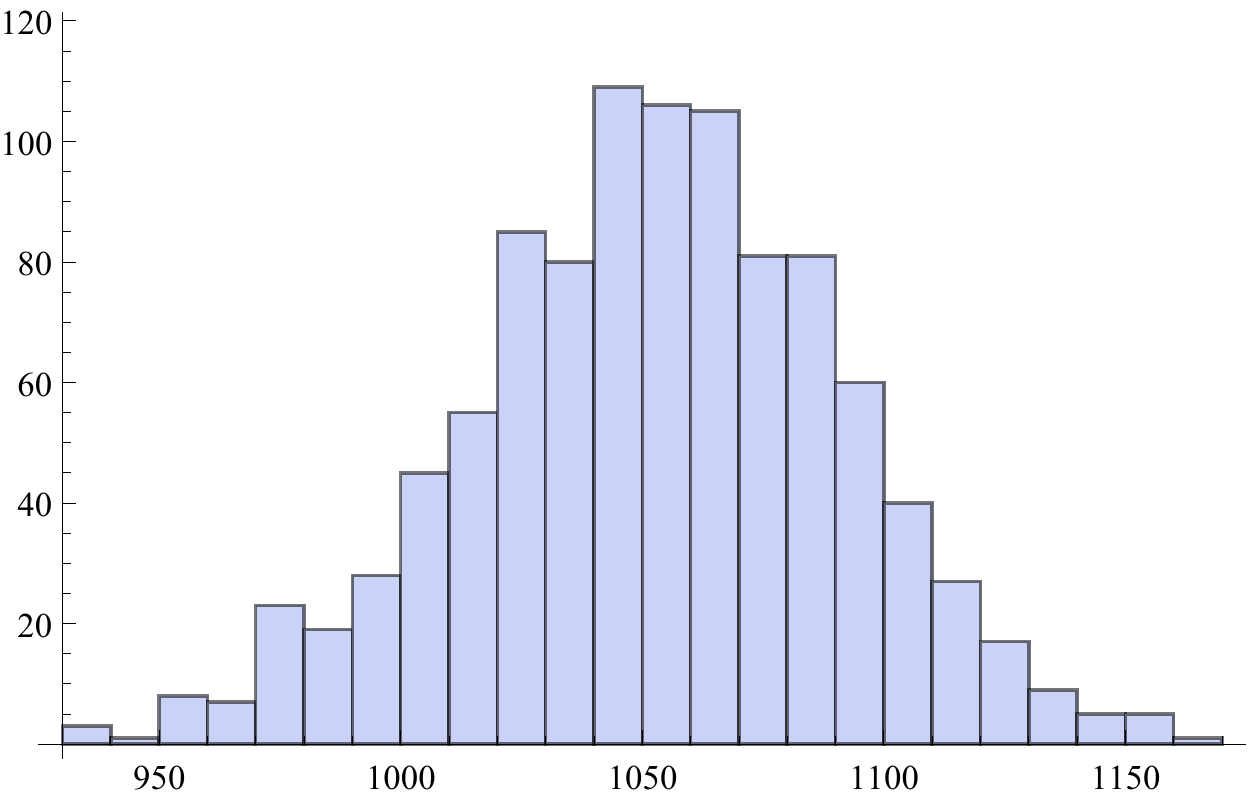}
\caption{\label{fig:histo500_900} Distribution of volume for random splittings of genus 5 and combinatorial length $900.$}
\end{figure}
\begin{figure}
\centering
\includegraphics[width=0.5\textwidth]{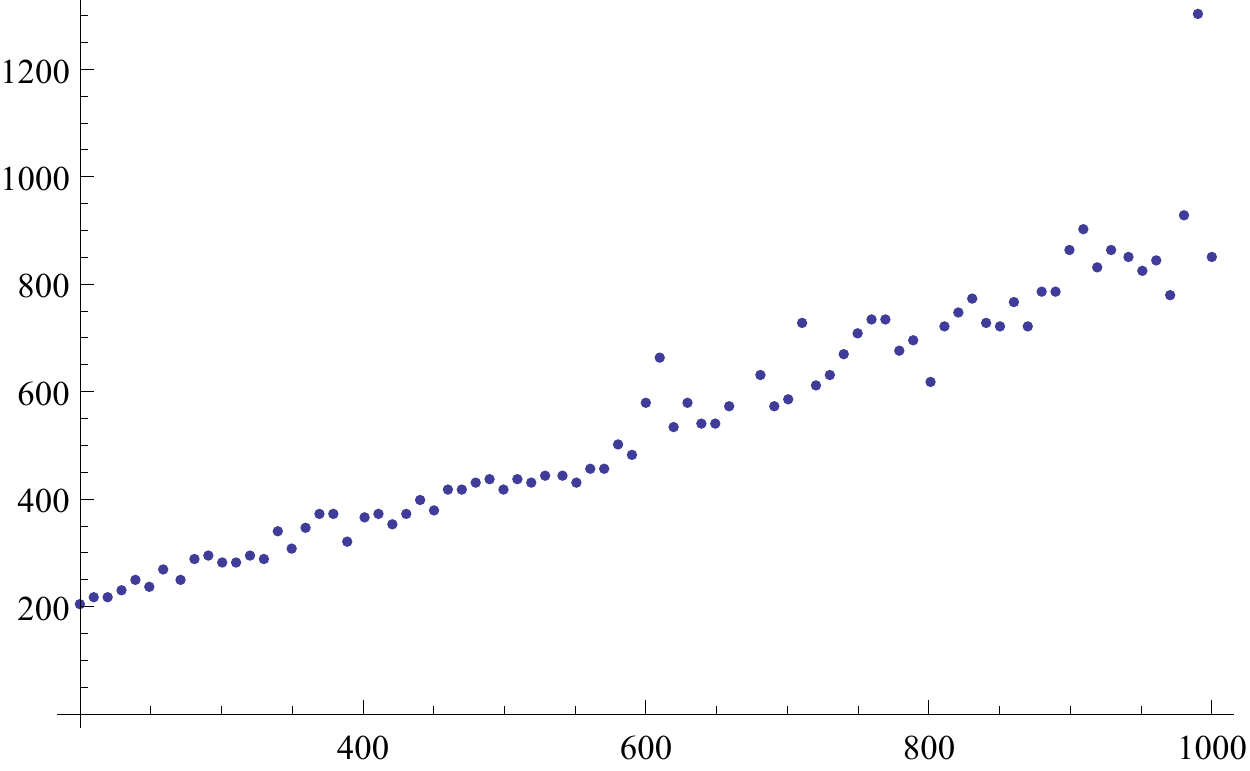}
\caption{\label{fig:graphvar2} Variance of volume for random splittings of genus 2}
\end{figure}
\begin{figure}
\centering
\includegraphics[width=0.5\textwidth]{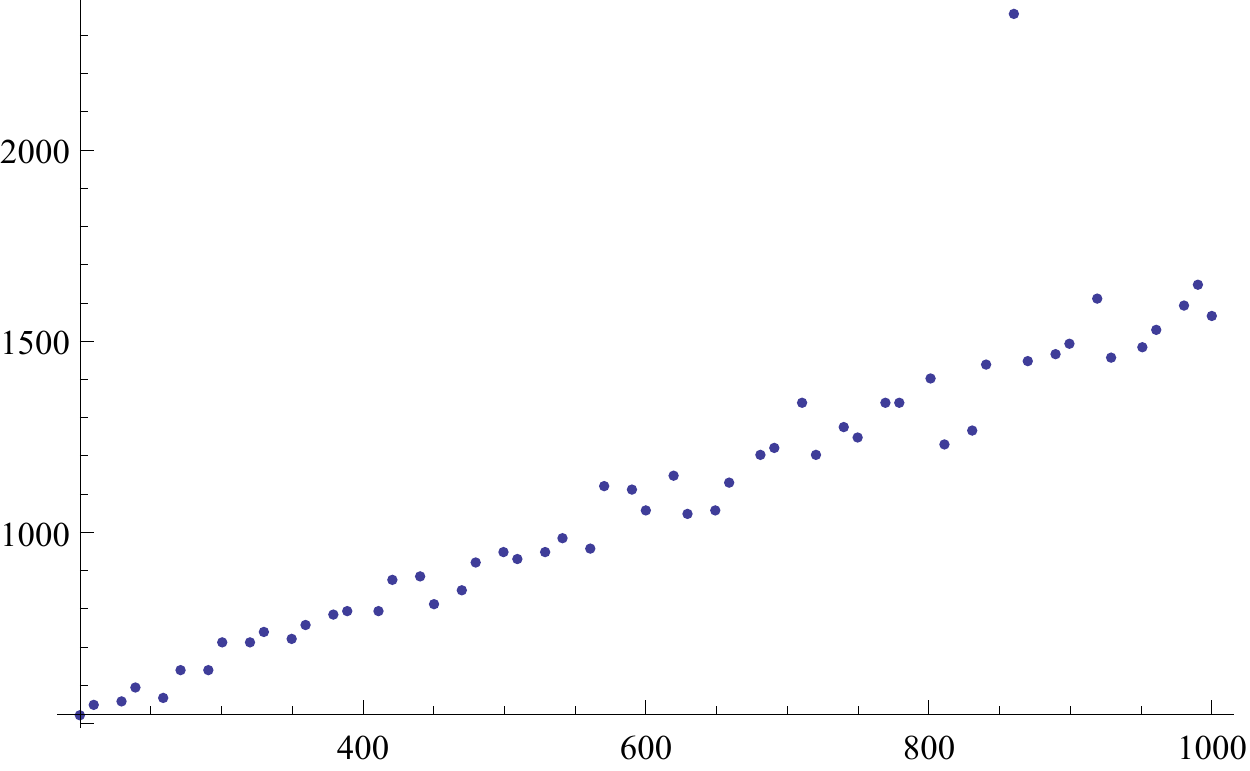}
\caption{\label{fig:graphvar5} Variance of volume for random splittings of genus 5}
\end{figure}
\newpage
\section{In conclusion}
It is fair to say that we now have good (though far from perfect) understanding of both random mapping tori and random Heegaard splittings. The first question is: what do we do with it? In combinatorics, the probabilistic method uses the understanding of random objects to use them as building blocks to construct particularly interesting examples. We have not yet done this in the setting of $3$-dimensional manifolds, but it is not inconceivable that this would be a promising way to go about proving the virtual fibering theorem, which would be, quite likely, more constructive than the Agol-Wise machine (this was almost certainly the intent of Dunfield and Thurston, but new ideas are needed).

It is also clear that at the moment the most natural model of random manifolds (\emph{via} random triangulations) is completely out of reach -- it is not even clear what tools are needed. There has been some very nice work on random complexes and hypergraphs (see \cite{linial2006homological,babson2011fundamental,aronshtam2013collapsibility,hoffman2013threshold},) but the methods clearly do not extend to the manifold setting (even to the setting of two-dimensional manifolds, where the theory of random triangulations is quite well-developed).
\bibliographystyle{plain}
\bibliography{msri}
\end{document}